\documentclass[12pt]{amsart}
\usepackage{graphicx}
\usepackage{amssymb}
\usepackage[utf8]{inputenc}
\usepackage{amsmath}
\usepackage{amsthm}
\usepackage{a4wide}
\usepackage{xcolor}
\usepackage{hyperref}

\newtheorem{definition}{Definition}[section]

\newtheorem{lemma}{Lemma}[section]
\newtheorem{theorem}{Theorem}[section]
\newtheorem{remark}{Remark}

\def\rr{\mathbb{R}}

\def\eps{\varepsilon}

\def\ml{\mathcal{L}}

\def\mE{\mathcal{E}}
\def\mL{\mathcal{L}}
\def\mH{\mathcal{H}}
\def\ll{\mathcal{L}_\lambda}
\def\el{\mathcal{E}_\lambda}

\title[]{A non-local coupling model involving three fractional Laplacians}

\author[A. Garriz  \and L. I. Ignat ]{ Alejandro Garriz \and Liviu I. Ignat }

\address{ A. G\'arriz
\hfill\break\indent Instituo de Matem\'aticas de la Universidad de Granada\\
Calle Ventanilla, 11, 18001 Granada SPAIN\\
28049 Madrid SPAIN
}
\email{{\tt
alejandro.garrizmolina@gmail.com}}

\address{L. I. Ignat
	\hfill\break\indent Institute of Mathematics ``Simion Stoilow'' of the Romanian Academy,
	Centre Francophone en Math\'{e}matique
	\\21 Calea Grivitei Street \\010702 Bucharest, ROMANIA.
}
\email{
	{\tt liviu.ignat@gmail.com}  
	\hfill\break\indent {\it Web page: }
	{\tt http://www.imar.ro/\~\,lignat}}

\thanks{This article was developed during a stay of  A.~G\'arriz at the "Simion Stoilow" Institute of Mathematics of the Romanian Academy with Liviu Ignat that was supported by the Agence Universitaire de la Francophonie (CA).
A. G\'arriz was also supported by the Spanish Ministry of Science and Innovation,  through projects  MTM2017-87596-P and SEV-2015-0554, by the Spanish National Research Council, through project 20205CEX001 and by the Institute of Mathematical Sciences ICMAT.}

\begin{document}

\keywords{Nonlocal diffusion, compactness arguments, gradient flow, asymptotic behavior, fractional Laplacian\\
\indent 2000 {\it Mathematics Subject Classification.} 35B40,
 45G10, 46B50.}

\begin{abstract}

In this article we study a non-local diffusion problem that involves three different fractional Laplacian operators acting on two domains. Each domain has an associated operator that governs the diffusion on it, and the third operator serves as a coupling mechanism between the two of them. The model proposed is the gradient flow of a non-local energy functional. In the first part of the article we provide results about existence of solutions and the conservation of mass. The second part encompasses results about the $L^p$ decay of the solutions. The third part is devoted to study the asymptotic behavior of the solutions of the problem when the two domains are a ball and its complementary.
Exterior fractional Sobolev and Nash inequalities of independent interest are also provided in an appendix.
\end{abstract}

\maketitle

\section{Introduction}

If one considers a non-local diffusion equation, probably one of the most famous and more deeply studied is the fractional heat equation, which can be written, formally, as
\begin{equation} \label{eq:fractional-heat}
u_t(x,t) + (-\Delta)^r u(x,t) = u_t + C_{N,r}P.V.  \int_{\rr^N} \frac{u(x)-u(y)}{|x-y|^{N+2r}}dy =0,
\end{equation}
for an $r\in (0,1)$, where $N$ is the dimension of the space, $C_{N,r}$ is a constant depending on $N$ and $r$ and $P.V.$ means that the integral must be understood in the Cauchy's Principal Value sense. Throughout the article we will omit these two letters refering the principal value for the sake of simplicity but the reader must recall that every single integral is understood in this sense when necessary. This equation is naturally associated with the energy
$$
\label{eq:energy.heat}
E(u) =  \frac{1}{4}\int _{\rr^N}\int_{\rr^N}  \frac{(u(x)-u(y))^2}{|x-y|^{N+2r}}\ d yd x
$$
in the sense that \eqref{eq:fractional-heat} is the $L^2$ gradient flow associated to $E(u)$. This equation models non-local diffusion derived from Lévy processes when the probability of particles jumping from point $x$ to point $y$ is given by the kernel $|x-y|^{-N-2r}$, which is a symmetric but singular function.

One limitation of this model is that it considers the ambient space as uniform, so it is natural to think about a model where the ambient space produces a different diffusion depending on which part of it the particle is in. To this end, an obvious possibility is to study the equation
\begin{equation}\notag
\begin{cases}
\displaystyle u_t(x,t)=\alpha_s\int _{\Omega_s} \frac{u(y,t)-u(x,t)}{|x-y|^{N+2s}}dy ,\qquad&x\in \Omega_s,\ t>0,\\

\displaystyle u_t(x,t)= \alpha_r\int _{\Omega_{r}} \frac{u(y,t)-u(x,t)}{|x-y|^{N+2r}}dy,\qquad&x\in \Omega_{r},\ t>0,\\

u(x,0)=u_0(x)\qquad&x\in \Omega,
\end{cases}
\end{equation}
for a couple of values $r,s\in (0,1)$ (the values $\alpha_s$ and  $\alpha_r$ are just, for now, normalization constants),  with $\Omega_s$ and $\Omega_{r}$ two disjoint open domains and $\Omega=\Omega_r\cup\Omega_s$
but this is a naive aproximation to the problem, since under this definition both domains (which we do not assume to be close to each other) are independent and thus the solutions of the equation must be studied separatedly by splitting the domain in two parts. No particle is allowed to have any information of what is happening in the domain it is not in. One possible way to solve this lack of intertwining is to ``couple'' the domains by considering ``jumps'' of the particles from $\Omega_s$ to $\Omega_r$ governed by a third non-local operator, another fractional Laplacian. In other words,
\begin{equation}\label{eq:main}
\begin{cases}
\displaystyle u_t(x,t)=\alpha_s\int _{\Omega_s} \frac{u(y,t)-u(x,t)}{|x-y|^{N+2s}}dy + \alpha_c\int _{\Omega_{r}} \frac{u(y,t)-u(x,t)}{|x-y|^{N+2c}}dy,&x\in \Omega_s,\ t>0,\\

\displaystyle u_t(x,t)= \alpha_r\int _{\Omega_{r}} \frac{u(y,t)-u(x,t)}{|x-y|^{N+2r}}dy + \alpha_c\int _{\Omega_s} \frac{u(y,t)-u(x,t)}{|x-y|^{N+2c}}dy,&x\in \Omega_{r},\ t>0,\\
u(x,0)=u_0(x),&x\in \Omega,
\end{cases}
\end{equation}
for $r,s, c\in (0,1)$. The reader must note that there are other possible options to couple this problem, see for example~\cite{MR3642095}, but the reason why we chose this one is mainly because this equation is the gradient flow of the energy functional
\begin{align}\notag
\label{energy.1}
E(u)&=\frac{\alpha_s}{4} \int _{\Omega_{s}}\int _{\Omega_{s}} \frac{(u(x)-u(y))^2}{|x-y|^{N+2s}}\ d yd x+\frac{\alpha_r}{4} \int _{\Omega_{r}}\int _{\Omega_{r}} \frac{(u(x)-u(y))^2}{|x-y|^{N+2r}}\ d yd x\\
&\quad+\frac{\alpha_c}{2} \int _{\Omega_{s}}\int _{\Omega_{r}} \frac{(u(x)-u(y))^2}{|x-y|^{N+2c}}\ d yd x,
\end{align}
as we shall see in the next section. Other possible reason is that it is interesting to consider the case where the probability of jumping to your own domain is not the same as the probability of jumping to the other domain, hence the difference in the exponents of the integration kernels. This kind of coupling problems have already been studied in~\cite{MR3642095, Kri} in the case of a different coupling method (in the so-called \textit{transmision problems}) and in~\cite{MR4114263} for different operators, the usual heat operator and another one given by a convolution with a probability kernel.

It is also important to note something about the constants $\alpha_s, \alpha_r, \alpha_c$. In our equations they already encompass, for simplicity, the constant originating from the fractional Laplacian operator, the one we called $C_{N,r}$ in~\eqref{eq:fractional-heat}. In other words,
\begin{equation}\label{eq:constants}
	\alpha_r = C_{N,r} C_{\Omega_{r}},
\end{equation}
where $ C_{\Omega_{r}}$ depends only on the domain $\Omega_{r}$ and on the value of $r$, and respectively for $c,s$. We use this notation because in the future we would like to study this problem when some of the values $r,s,c$ go to one or to zero, see~\cite{MR4114263}. In both questions these constants are expected to play an important role when determining the equation obtained in the limit, but in the present one not so much in the proofs, so we will omit these three constants most of the times for the sake of simplicity. If the reader desires so he can consider $C_{\Omega_{r}}=1$  for a tidier presentation of the article.

The gradient flow structure of this problem already provides a certain $L^2$ existence theory, but in order to study the problem in $L^1$ we make use of semigroups.
\begin{theorem}\label{thm:existence_banach}
For any $u_0\in L^1(\Omega)$ there exists a unique solution  $u\in C([0,\infty),L^1(\Omega))$ to the problem~\eqref{eq:main}. It conserves the mass and its $L^1(\Omega)$ norm does not increase in time  
$$
\|u(t_2)\|_{L^1(\Omega)} \leq  \|u(t_1)\|_{L^1(\Omega)}\quad \text{for any }0\leq t_1\leq t_2 < \infty.
$$
\end{theorem}
Keep in mind that depending on the initial datum the existing theory allows us to improve on the properties of the solutions, as we shall see.

%
%

\medskip 

Now once the existence of solutions of this problem is stablished properly, we would like to discuss the case when $\Omega_s$ is a Lipschitz bounded domain and $\Omega_r=\rr^N\setminus \overline{\Omega_s}$ satisfies the following measure density condition:
there exists a positive constant $C_{\Omega_r}$ such that
 \begin{equation}
\label{measure.condition}
  |\Omega_r\cap B_\rho(x)|\geq C_{\Omega_r} \rho^N, \quad\forall x\in \Omega_r,  \forall \rho>0.
\end{equation}
When this condition is satisfied for all $\rho\in (0,1]$ the domain $\Omega_r$ is called \textit{regular} or \textit{plump}. Under these assumptions we want to study the competition between diffusions when one domain is bounded and the other one, its complementary, is an exterior domain, and how this competition determines the shape of the solution for big times.

%
%

\begin{theorem}
	\label{decay}For any $u_0\in L^1(\rr^N) $ and $1\leq p\leq \infty$ there exists a positive constant $C=C(p,N,\Omega_s,\Omega_r)$ such that the 
	 solution of system \eqref{eq:main} satisfies
	 $$
		\label{decay-rate}
		 \|u(t)\|_{L^p(\rr^N)}\leq C( t^{-\frac N{2r}(1-\frac 1p)}+ t^{-\frac N{2\min\{r,s\}}(1-\frac 1p)})\|u_0\|_{L^1(\rr^N)}, \quad \forall t>0.
	$$
The above behavior  is optimal in the class of the polynomial grow/decay.
\end{theorem}

%
%
%
%
%

Let us say a few words about the optimality of the above powers of the time variable. In fact any $\alpha>0$ such that 
\begin{equation}\label{l1-linfty}
 \|u(t)\|_{L^\infty(\rr^N)}\leq C  t^{-\frac N{2\alpha}(1-\frac 1p)}\|u_0\|_{L^1(\rr^N)},\quad  \forall u_0\in L^1(\rr^N),
 \end{equation}
holds for all time $t\in (0,1)$ 
must satisfy
$
\alpha\leq \min\{r,s\}.
$
Also, if \eqref{l1-linfty} holds for all $t>t_0$ then $\alpha\geq r$. We emphasize that the above optimality is in the class of polynomial grow/decay and it does not exclude the possible logarithmic corrections at zero of infinity. However, when $r<N/2+c$ we will prove in the next theorem that the behaviour at infinity cannot be improved and the polynomial decay is the correct one.
In dimension $N\geq 2$, we always have $N/2+c>1>r$ thus $\alpha\geq r$. In dimension one it remains to analyze what is the optimal decay when $1/2+c<r<1$.

We can be more precise about the limit profile when the time goes to infinity. To avoid geometrical technicalities we consider the particular case when $\Omega_s$ is the unit ball $B_1(0)$ and $\Omega_r$ is its open complementary. 
 We  expect the mass to accumulate in the unbounded domain and thus the solution must look like the solution of problem~\eqref{eq:fractional-heat} with initial datum the Dirac's delta  times the mass of the initial datum and fractional exponent the one corresponding to the unbounded domain.  
This is  true in a certain range of the exponents $r,s$ and $c$.
\begin{theorem}\label{first.term}Let $\Omega_s=B_1(0)$,  $\Omega_r=\mathbb{R}^N\setminus \overline{B_1(0)}$ and 
$r,s$ and $c$ satisfying  
$$
\label{hip.exp.strict}
  r<\frac N2+c.
$$
For any $u_0\in L^1(\rr^N) $ the solution of equation~\eqref{eq:main}   satisfies
$$
\lim_{t\to \infty} t^{\frac{N}{2r}\left(  1-\frac{1}{p}  \right)}\|u(t) - U_M(t)\|_{L^p(\rr^N)}  = 0\quad\text{for all }1\leq p< \infty,
$$
where $U_M(x,t)$ is the solution of equation
$$
u_t + C_{\Omega_r}(-\Delta)^r u = 0\quad \text{for all }x\in\mathbb{R}^N,\qquad u(x,0)=M\delta_0,
$$
$\delta_0$ is Dirac's delta centered at $x=0$ and $M$ is the mass of the initial data.
\end{theorem}

\begin{remark}
	It is not hard to show that $U_M(x,t)=MC_{\Omega_r}^{\frac 1{2r}} K_t^r(C_{\Omega_r}^{\frac 1{2r}}x)$, where $K_t^r(x)$ is the solution of the same equation when $C_{\Omega_r}=1$.
\end{remark}

The article is divided as follows. In Section 2 we address the issues of gradient flow structure, existence and conservation of mass. Section 3 focuses on the $L^p$ decay of the norms of the solution. Finally Section 4 deals with the first term in the asymptotic expansion for large time of the solution when the problem is posed in the unit ball and its complementary. There is an Appendix in Section 5 where we prove the Sobolev inequality for \textit{regular}~\cite{MR3280034} exterior domains (also called plump domains) and an exterior $L^p$ Nash inequality.

\section{Gradient flow structure and existence of solutions}

Let us consider two disjoint domains $\Omega_s$ and $\Omega_r$. We introduce the spaces
$$
X^c(\Omega_{s},\Omega_{r}):=\left\{u\in L^2(\Omega_s\cup \Omega_r) : \int_{\Omega_{s}}\int_{\Omega_{r}} \frac{(u(x)-u(y))^2}{|x-y|^{N+2c}} d yd x<\infty\right\}
$$
and
$$
\mathcal{H}(\Omega_s\cup \Omega_r):= H^s(\Omega_{s})\cap H^r(\Omega_{r})\cap X^c(\Omega_{s},\Omega_{r})
$$
where $H^s$ is the usual $s$-fractional Sobolev space. The exponents satisfy $s,r,c\in (0,1)$. Now we take three positive constants $\alpha_s,\alpha_r$ and $\alpha_c$ depending on their subindexes and on the dimension $N$ of the space and define  the energy by~\eqref{energy.1} if $u\in \mathcal{H}(\Omega_s\cup \Omega_r)$
and $E(u)=\infty$ if not, with its associated  form $\mE:\mathcal{H}(\Omega_s\cup \Omega_r)\times \mathcal{H}(\Omega_s\cup \Omega_r)\rightarrow \rr$, defining with it the associated inner product $<u,v>:=\mathcal{E}(u,v)$. This energy functional is proper, convex and lower semi-continuous, so following, for an instance, \cite[9.6.3, Thm 4]{MR2597943} we define the operator that will be in the end the subdifferential as
\begin{align}\notag
\label{subdifferential.operator}
\mathcal{A}[u](x):= - \int_\Omega\Big\{ & \alpha_s\frac{u(y)-u(x)}{|x-y|^{N+2s}} \chi_{\Omega_s}(y) \chi_{\Omega_s}(x)  + \alpha_r\frac{u(y)-u(x)}{|x-y|^{N+2r}} \chi_{\Omega_r}(y) \chi_{\Omega_r}(x) \\
\nonumber& \alpha_c\frac{u(y)-u(x)}{|x-y|^{N+2c}} \chi_{\Omega_s}(y) \chi_{\Omega_r}(x)  + \alpha_c\frac{u(y)-u(x)}{|x-y|^{N+2c}} \chi_{\Omega_r}(y) \chi_{\Omega_s}(x) \Big\} dy
\end{align}
with domain $D(\mathcal{A}):=\{ u\in L^2(\Omega): \mathcal{A}[u]\in L^2(\Omega)\}$ endorsed with the usual norm, $\| u\|^2_{D(\mathcal{A})}:= \|u\|^2_{L^2(\Omega)} + \|\mathcal{A} [u]\|^2_{L^2(\Omega)}$.

First we need to prove that $D(\mathcal{A})\subseteq D(\partial E)$ and to this end the only difficult point is to prove that given $u\in D(\mathcal{A})$ and defining $v=\mathcal{A}[u]$ we  have that $\mE(v, w-u) \leq E[w] - E[u]$ for every $w\in \mathcal{H}(\Omega)$. This is easy once we see that
\begin{equation}\label{eq:propiedad_Laplaciano_fraccionario}
- \int _{B}\int _{C} \frac{u(y)-u(x)}{|x-y|^{N+2c}} f(x)\ d yd x = \int _{C}\int _{B} \frac{u(y)-u(x)}{|x-y|^{N+2c}} f(y)\ d yd x
\end{equation}
which implies
\begin{equation}\label{eq:propiedad_Laplaciano_fraccionario_simetrico}
 - \int _{B}\int _{B} \frac{u(y)-u(x)}{|x-y|^{N+2c}} f(x)\ d yd x = \frac{1}{2} \int _{B}\int _{B} \frac{u(y)-u(x)}{|x-y|^{N+2c}} (f(y)-f(x))\ d yd x
\end{equation}
whenever this integrals are well defined for general domains $B,C$, a general function $f$ and $0<c<1$. In our case the chosen domains and the Cauchy-Schwarz Inequality ensure the integrals are well defined. It is also helpful to note that
$$
ab-b^2= \frac{a^2-b^2-(a-b)^2}{2}\leq \frac{a^2-b^2}{2}
$$
for any real numbers $a,b$. The next and final part is to see that $D(\mathcal{A})\supseteq D(\partial E)$, so to this point we take a $f\in L^2(\Omega)$ and look for a minimizer $J:\mathcal{H}(\Omega)\to \mathbb{R}$ defined by
$$
J[w]=E[w] + \int_\Omega \frac{w^2(x)}{2} -f(x)\cdot w(x)\ dx.
$$
This minimizer is precisely a function $u\in D(\mathcal{A})$ that satisfies $u+\mathcal{A}[u]=f$ weakly in $\Omega$, which also shows via well known arguments that $u\in D(\mathcal{A})$ since $\|\mathcal{A}[u]\|_{L^2(\Omega)}\leq ||f||_{L^2(\Omega)}$. Consequently the range of $I+\mathcal{A}$ is $L^2(\Omega)$ and this implies, see~\cite[Section 9.6.3, Thm 4]{MR2597943}, that $D(\mathcal{A})\supseteq D(\partial E)$.


%
%
Therefore, by \cite[9.6.3, Thm 3]{MR2597943}, we say that this equation is provided by the gradient flow of the energy $E[u]$, and in this sense it motivates our study. 
%
%
%

We now provide the existence results for our problem. It is based on the classical semigroup theory. The bilinear form $\mE:\mathcal{H}(\Omega_s\cup \Omega_r)\times \mathcal{H}(\Omega_s\cup \Omega_r)\rightarrow \rr$ defined by 
\[
\mE(u,v)=\frac{E(u+v)+E(u-v)}2
\]
is a densely defined, accretive, continuous and closed sesquilinear form on $L^2(\Omega_s\cup \Omega_r)$. It introduces the unbounded operator $-\mL:D(\mL)\subset L^2(\Omega_r\cup \Omega_s)$ by 
\[
\mE(u,\varphi)=(-\mL[u],\varphi)_{L^2(\Omega_r\cup \Omega_s)} \quad \forall \varphi\in \mH.
\]
Note that $\mathcal{L}=-\mathcal{A}$. It follows \cite[Prop.~1.5.1, p.29]{MR2124040} that $\mL$ generates a strongly continuous  semigroup of contraction $e^{t\mL}$ on $L^2(\Omega_r\cup \Omega_s)$. In particular since $\mE$ is symmetric, $\mL$ is self-adjoint and the following result holds \cite[Theorem 3.2.1]{MR1691574}.
\begin{theorem}\label{thm:existence_hilbert_CH}
For any  $u_0\in L^2(\Omega_r\cup \Omega_s)$ there exists a unique solution $u(t)=e^{t\mL}u_0\in C([0,\infty),L^2(\Omega))$ of problem \eqref{eq:main}. 
Moreover,  
\[u\in C([0,\infty):L^2(\Omega))\cap C((0,\infty):D(\mathcal{L}))\cap C^1((0,\infty):L^2(\Omega))
\]
and 
$$
\mE(u(t),u(t))\leq \frac{1}{2t}\|u_0\|^2_{L^2(\Omega)}, \forall t>0.
$$
\end{theorem}

We now explain how this result can be extended to the case of $L^p(\Omega)$ solutions, $1\leq p\leq \infty$ and prove Theorem \ref{thm:existence_banach}. The strategy is classical (see for example \cite{MR2124040}).
First,  for any $u\in \mH(\Omega_s\cup \Omega_r)$ we have $|u|\in \mH(\Omega_s\cup \Omega_r)$ and $a(|u|,|u|)\leq a(u,u)$. This implies \cite[Corrolary 2.18]{MR2124040} that the semigroup is  positive: $u_0\geq 0 $ implies $e^{t\mL}u_0\geq 0$. In particular, comparation principle holds: $u_0\leq v_0$ then $e^{t\mL}u_0\leq e^{t\mL}u_0$. Also, since for any $u\in \mH(\Omega_s\cup \Omega_r)$ with $u\geq 0$ we have $1\wedge u=\max\{1,u\}\in \mH(\Omega_s\cup \Omega_r)$ and $a(1\wedge u,1\wedge u)\leq a(u,u)$ the semigroup is $L^\infty(\Omega)$ contractive. Once this has been established we use that $\mL$ generates a strongly continuous semigroup of contractions on $L^2(\Omega)$ to obtain that it also generates a strongly continuous semigroup of contractions on $L^p(\Omega)$ for any $2\leq p<\infty$. By duality, since $\mL$ is self-adjoint the same property transfers to   $L^p(\Omega)$ for any $1< p \leq 2$. The $L^1(\Omega)$ case follows by a density argument (see \cite{MR2124040}, p.56). The above consideration gives the existence and contraction property in  Theorem \ref{thm:existence_banach}. It remains to prove the mass conservation.
In view of the $L^1$-contraction property it is sufficient to consider the  case of initial data  $u_0\in L^1(\Omega)\cap L^2(\Omega)$.
 When $\Omega$ is bounded we use that the function identical one, $1_{\Omega}$, belongs to $\mH(\Omega_s\cup \Omega_r)$, $u(t)\in D(\mathcal {L})$ and then
\[
\frac{d}{dt}\int_{\Omega} u(x,t)dx=(\mL u(t),1_{\Omega})=-\mE(u(t),1_{\Omega})=0.
\]
When $\Omega$ is unbounded we consider an approximation of the identity, a smooth function  $\chi_R\in [0,1]$ such that $\chi_R\equiv 1$ in $|x|<R$ and $\chi_R\equiv 0$ in $|x|>2R$. Using that $\mE(\chi_R,\chi_R)\rightarrow 0$ as $R\rightarrow \infty$ and dominated convergence theorem show give us the conservation of the mass. This finishes the proof of \ref{thm:existence_banach}.

\section{The decay of the solutions}
As we said in the Introduction we consider 
$\Omega_s$ to be a Lipschitz bounded domain such that its complementary $\Omega_r=\rr^N\setminus \overline{\Omega_s}$ satisfies the measure density condition \eqref{measure.condition}. The assumption on $\Omega_s$ guarantees the existence of classical Sobolev embeddings  as well as Gagliardo-Nirenberg-Sobolev inequalities \cite{MR3990737}. The assumption on the exterior domain $\Omega_r$ is sufficient to obtain a Sobolev inequality in Lemma~\ref{GNS.ext.balls} of the appendix. Using  this Sobolev inequality we prove Nash-like inequalities for exterior domains and use them to prove Theorem \ref{decay} regarding the long time decay of the solutions. 

To simplify the presentation for $0<s<1$ and $\Omega$ an open set, we will denote by $[f]_{s,\Omega}$ the following quantity  
\[
[f]_{s,\Omega}^2:=\int_{\Omega} \int_{\Omega}\frac{(f(y)-f(x))^2}{|x-y|^{N+2s}}dxdy.
\]

\begin{lemma}\label{lemma.nash.2}(Nash's inequality for exterior domains)
Let  $N\geq 1$, $r\in (0,1)$. For any $f\in H^r(\Omega_r)\cap L^1(\Omega_r)$ the following holds
\begin{equation}
\label{nash.2}
  \|f\|_{L^2(\Omega_r)}\leq C(\Omega_r,r,N) \|f\|_{L^1(\Omega_r)}^{\frac{2r}{N+2r}}[f]_{r,\Omega_r}^{\frac N{N+2r}}.
\end{equation}
\end{lemma}

\begin{proof}
	For $N\geq 2$ we have $N>2r$  and we use H\"older's inequality and Sobolev inequality for exterior domains in  Lemma \ref{GNS.ext.balls}:
	\[
	 \|f\|_{L^2(\Omega_r)}\leq \|f\|_{L^1(\Omega_r)}^{\frac{2r}{N+2r}}
\|f\|_{L^{2^*}(\Omega_r)}^{\frac{N}{N+2r}} 	 \leq C(\Omega_r,r,N) \|f\|_{L^1(\Omega_r)}^{\frac{2r}{N+2r}}[f]_{r,\Omega_r}^{\frac N{N+2r}}.
	\]
	In dimension $N=1$ we consider $\Omega_{r}=(-\infty,a)\cup(b,\infty)$, where $-\infty\leq a<b\leq\infty$, and write $f=f_-+f_{+}$ where $f_-$ and $f_{+}$ are the restrictions of $f$ to $(-\infty, a)$ and $(b,\infty)$ respectively. Since 
	\[ 
	[f_-]_{r,(-\infty,a)}^2+[f_+]_{r,(b,\infty)}^2\leq [f]_{r,\Omega_r}^2,
	\]
	it is sufficient to prove the above estimate only for $f_+$ and the corresponding interval $(b,\infty)$. After a  translation of the interval to the origin we have to prove that for any $f\in H^r(\rr_+)\cap L^1(\rr_+)$ it holds
	\[
	  \|f\|_{L^2(\rr_+)}\leq C(\Omega_r,r,N) \|f\|_{L^1(\rr_+)}^{\frac{2r}{1+2r}}[f]_{r,\rr_+}^{\frac 1{1+2r}}.
	\]
Let us consider $f_{even}$, the even extension of $f$.
Using \cite[Lemma 5.2]{MR2944369} we obtain that $f_{even}\in H^r(\rr)$ and $[f_{even}]_{r,\rr}\leq 2 [f]_{r,\rr_+}$. It is then sufficient to prove the inequality for functions defined on the whole line:
 \[
	  \|f\|_{L^2(\rr)}\leq C(\Omega_r,r,N) \|f\|_{L^1(\rr)}^{\frac{2r}{1+2r}}[f]_{r,\rr}^{\frac 1{1+2r}}.
	\]
This inequality holds in view of  \cite[Theorem~1.3]{MR2299447} which translates inequalities for the Laplacian to fractional Laplacian with the corresponding exponent.
\end{proof}

\begin{lemma}
	There exists a positive constant $C=C(\Omega_s,\Omega_r,N,c,r)$ such that 
	\begin{equation}
	\label{ineq.energy.p}
		 \|u\|^2_{L^2(\rr^N)} \leq C\Big(  \mE(u,u)+\|u\|_{L^1(\Omega_r)}^{\frac{4r}{N+2r}}\mE(u,u)^{\frac{N}{N+2r}}\Big).
	\end{equation}
	holds for all $u\in L^1(\rr^N)\cap \mathcal{H}(\rr^N)$. 
\end{lemma}

\begin{proof}
 We have that $\mE(u,u)=E_1(u)+E_2(u)$, where
\[
E_1(u)=\frac{\alpha_s}{4} \int _{\Omega_s}\int _{\Omega_s} \frac{(u(x)-u(y))^2}{|x-y|^{N+2s}}\ d yd x+ \frac{\alpha_c}{2} \int _{\Omega_s}\int _{\Omega_r} \frac{(u(x)-u(y))^2}{|x-y|^{N+2c}}\ d yd x
\]
and
\[
E_2(u)=\frac{\alpha_r}{4} \int _{\Omega_r}\int _{\Omega_r} \frac{(u(x)-u(y))^2}{|x-y|^{N+2r}}\ d yd x.
\]

We estimate each of the above terms. In the case of $E_1$ we 
estimate it from bellow in term of the $L^2$-norms of $u$. Recall the following elementary inequality
\[
(a-b)^2=a^2-2ab+b^2\geq a^2-(1-\varepsilon) a^2-\frac{1}{1-\varepsilon}b^2+b^2=\eps(a^2-\frac 1{1-\varepsilon}b^2).
\]
Let us choose $R=R(\Omega_{s})>0$ such that $\Omega_s\subset B_R(0)$ and  $\varepsilon=1/2$. We get
\begin{align*}
\label{}
  E_1(u)&\geq \frac{\alpha_c}{2} \int _{\Omega_s} \int_ {\Omega_r} \frac{(u(x)-u(y))^2}{|x-y|^{N+2c}}dxdy
\geq \frac{\alpha_c}{2}   \int _{\Omega_s} \int_ {\Omega_r} \frac{(u(x)-u(y))^2}{(1+|x-y|)^{N+2c}}dxdy\\
  &\geq \frac {\alpha_c}4\int _{\Omega_s} \int_ {\Omega_r} \frac{u^2(x)-2u^2(y)}{(1+|x-y|)^{N+2c}} dxdy\\
  &= \frac {\alpha_c}4\int_{\Omega_s}u^2(x)\int _{\Omega_r}\frac {dy}{(1+|x-y|)^{N+2c}} dx-\frac{\alpha_c}{2}
  \int_{\Omega_r}u^2(y) \int _{\Omega_s}\frac{dx}{(1+|x-y|)^{N+2c}}dy\\
 &\geq  \frac {\alpha_c}4 \int_{\Omega_s}u^2(x)dx\int _{|y|>R}\frac {dy}{(2+|y|)^{N+2c}}-\frac{\alpha_c}{2}
  \int_{\Omega_r}u^2(y) dy\int _{|x|<R}dx \\
  &= \frac {\alpha_c}4 C(R,N,c) \int_{\Omega_s}u^2(x)dx -C(R,N)\frac{\alpha_c}{2}  \int_{\Omega_r}u^2(y) dy\\
  &\geq C(R,N,c) ( \|u\|^2_{L^2(\Omega_s)}-\|u\|^2_{L^2(\Omega_r)}).
\end{align*}
%
%
%
%
%
%
%
%
%
Using  Nash inequality for exterior domains \eqref{nash.2}
it follows that the $L^2(\rr^N)$ norm of $u$ satisfies
\begin{align*}
   \|u\|^2_{L^2(\rr^N)}&\leq C(\Omega_s, N, c) E_1(u)+  \|u\|^2_{L^2(\Omega_r)}\\
   &\leq C(\Omega_s, N, c) E_1(u)+ C(\Omega_{r},  r, N)\|u\|_{L^1(\Omega_r)}^{\frac{4r}{N+2r}}(E_2(u))^{\frac{N}{N+2r}},
\end{align*}
which finishes the proof.
\end{proof}

\begin{proof}[Proof of Theorem \ref{decay}]
\textbf{Step I. The behavior at $t=0$}.
We denote $\alpha=\min\{r,s\}$ and prove that
\begin{equation}\label{min.r.s}
\|v\|_{L^2(\rr^N)}^{2+\frac{4\alpha}N}\leq  C(\Omega_r,\Omega_s,r,s,N)\|v\|_{L^1(\rr^N)}^{\frac{4\alpha}N}(\|v\|^2_{L^2(\rr^N)}+\mE(v,v)), \quad \forall v\in \mathcal{H}(\rr^N)\cap L^1(\rr^N).
\end{equation}
We claim that 
\begin{equation}
\label{ineg.in.ball}
\|v\|_{L^2(\Omega_s)}^{2+\frac {4s}N}\leq C(\Omega_s,s,N)\|v\|_{L^1(\Omega_s)}^{\frac {4s}N}\|v\|_{H^{s}(\Omega_s)}^2.
\end{equation}
Using this inequality, Lemma \ref{lemma.nash.2} and the fact that $\|u\|_{L^1(\Omega_s)} \leq |\Omega_{s}|^{1/2}\|u\|_{L^2(\Omega_s)}$ we obtain that
\begin{align*}
\mE(v,v)+\|v\|_{L^2(\rr^N)}^2&\gtrsim [v]_{s,\Omega_s}^2+\|v\|_{L^2(\Omega_s)}^2+[v]_{r,\Omega_r}^2+\|v\|_{L^2(\Omega_r)}^2\\
&\gtrsim \|v\|_{L^2(\Omega_s)}^2\Big(\frac{\|v\|_{L^2(\Omega_s)}}{\|v\|_{L^1(\Omega_s)}}\Big)^{\frac{4s}N}+\|v\|^2_{L^2(\Omega_r)}\Big[ \Big(\frac{\|v\|_{L^2(\Omega_r)}}{\|v\|_{L^1(\Omega_r)}}\Big)^{\frac{4s}N}+1\Big]\\
&\gtrsim \|v\|_{L^2(\Omega_s)}^2\Big(\frac{\|v\|_{L^2(\Omega_s)}}{\|v\|_{L^1(\Omega_s)}}\Big)^{\frac{4\alpha}N}|\Omega_s|^{\frac{2(\alpha-s)}N}+
C(s,r)\|v\|_{L^2(\Omega_r)}^2\Big(\frac{\|v\|_{L^2(\Omega_r)}}{\|v\|_{L^1(\Omega_r)}}\Big)^{\frac{4\alpha}N},
\end{align*}
which gives us   inequality \eqref{min.r.s}.
In view of \cite[Th.~2.1]{MR898496} we obtain that the semigroup satisfies  
\begin{equation}
\label{est.1.infty.0}
\|S(t)u_0\|_{L^\infty(\rr^N)}\leq \frac{\beta e^{  t}}{t^{\frac N{2\alpha}}}\|u_0\|_{L^1(\rr^N)}\quad \forall u_0\in L^1(\rr^N)
\end{equation}
for a certain positive constant $\beta$. This estimate  shows that the solution belongs to all the spaces $L^p(\rr^N)$, $1\leq p\leq \infty$, whenever the initial datum belongs to $L^1(\rr^N)$. Using the $L^1$-contraction property one can obtain that the map $t\rightarrow \|S(t)\|_{L^1,L^\infty}$ is a decreasing function and then the above estimate can be improved for $t>1$ but without obtaining a decay for large times
\begin{equation}
\label{est.1.infty.1}
  \|S(t)u_0\|_{L^\infty(\rr^N)}\leq \frac{\beta'}{\min\{1,t\}^{\frac N{2\alpha}}}\|u_0\|_{L^1(\rr^N)}\quad \forall u_0\in L^1(\rr^N), \forall \ t>0.
\end{equation}
This gives the desired estimate for small times $t<1$.

It remains to prove claim \eqref{ineg.in.ball}. We use Sobolev embeddings and interpolation inequalities for the Lipschitz bounded domain  $\Omega_s$. We refer to \cite{MR3990737} for a similar approach in dimension two. The classical Sobolev embedding (see for example \cite[Theorem~B]{MR3990737}),
\[
\|v\|_{L^2(\Omega_s)}\leq C(\Omega_s,s,N) \|v\|_{W^{\frac{2sN}{4s+2N},\frac{4s+2N}{4s+N}}(\Omega_s)}
\]
and  Gagliardo-Nirenberg-Sobolev inequalities (see for example \cite[Corollary~1]{MR3990737}),
\[
\|v\|_{W^{\frac{2sN}{4s+2N},\frac{4s+2N}{4s+N}}(\Omega_s)}\leq C(\Omega_s,s,N) \|v\|_{W^{0,1}(\Omega_s)}^\theta \|f\|^{1-\theta}_{W^{s,2}(\Omega_s)},\ \theta=\frac{2s}{2s+N},
\]
give  the desired result.

Let us now prove that this blow-up at $t\downarrow 0$ cannot be improved. In fact \cite[Th.~2.1]{MR898496} shows that estimate \eqref{est.1.infty.0} is equivalent to the estimate
\begin{equation}
\label{est.nash.10}
\|v\|_{L^2(\rr^N)}^{2+\frac{4\alpha}N}\leq C\|v\|_1^{\frac{4\alpha}N}(\|v\|^2_{L^2(\rr^N)}+\mE(v,v)),\ \forall v\in \mathcal{H}(\rr^N)\cap L^1(\rr^N).
\end{equation}
 We prove  that any $\alpha>0$ for which the above inequality holds satisfies $\alpha\leq 
 \min\{r,s\}$. Choose any function $\varphi\in C_c^\infty(\rr^N)$ supported in the unit ball and define $\varphi_\eps(x)=\varphi((x-x_0)/\eps)$ for some point $x_0\in\rr^N$ and $\eps>0$ a small parameter. It follows that ${\rm{supp}}(\varphi_\eps)\subset B_\eps(x_0)$ and $\|\varphi_\eps\|_{L^p(\rr^N)}\simeq 
 \eps^{\frac Np}$. Introducing this estimate in \eqref{est.nash.10} we obtain that
 \[
 \varepsilon^{\frac N2(2+\frac {4\alpha}N)}\lesssim \varepsilon^{4\alpha}(\eps^N+\mathcal{E}(\varphi_\eps,\varphi_\eps))
  \]
and then
\[
\eps^{N-2\alpha}\lesssim \eps^N +\mathcal{E}(\varphi_\eps,\varphi_\eps).
\]

We now prove that choosing $x_0\in \Omega_s$ gives us that $\alpha\leq s$ while $x_0\in \Omega_r$ gives $\alpha\leq r$. 

Case I. Let us choose $x_0\in \Omega_s$ and $\eps>0$  such that $\eps<d(x_0,\Omega_r)/2$. Thus $B_\eps(x_0)\subset \Omega_s$, the support of $\varphi_\eps$ is contained in the ball $\Omega_s$ and $[\varphi_\eps]_{r,\Omega_r}=0$. It follows that  
\begin{align*}
	\mathcal{E}(\varphi_\eps,\varphi_\eps)&\leq \eps^{N-2s}[\varphi]_{s,\rr^N}+\int_{\Omega_s}\int_{\Omega_r}\frac{(\varphi_\eps(x)-\varphi_\eps(y))^2}{|x-y|^{N+2c}}dydx\\
	&\leq  \eps^{N-2s}[\varphi]_{s,\rr^N}+\int_{\Omega_s}\varphi^2_\eps(x)\int_{\Omega_r}\frac{1}{|x-y|^{N+2c}}dydx\\
	&\leq  \eps^{N-2s}[\varphi]_{s,\rr^N}+\int_{B_\eps(x_0) }\varphi^2(\frac {x-x_0}\eps)\int_{|y-x|>d(x_0,\Omega_r)/2}\frac{1}{|x-y|^{N+2c}}dydx\\
	&\lesssim \eps^{N-2s}[\varphi]_{s,\rr^N}+ \eps^N \int_{|x|<1}\varphi^2(x)dx.
\end{align*}
It implies that
\[
\eps^{N-2\alpha}\lesssim  \eps^N +\eps^{N-2s}, \quad \forall \eps\in (0,1), 
\]
which clearly implies $\alpha\leq s$.

Case II. Let now choose $x_0\in \Omega_r$ and $\eps>0$ small enough such that  $B_\eps(x_0)\subset \Omega_r$. Thus $ {\rm {supp}}\ (\varphi_\eps) \subset B_\eps(x_0)\subset \Omega_r$ and $[\varphi_\eps]_{s,\Omega_s}=0$. In this case the cross term is small similarly to the previous case
\begin{align*}
	\int_{\Omega_s}\int_{\Omega_r}&\frac{(\varphi_\eps(x)-\varphi_\eps(y))^2}{|x-y|^{N+2c}}dydx=
	\int_{\Omega_r}\varphi^2_\eps(y)\int_{\Omega_s}\frac{ 1}{|x-y|^{N+2c}}dxdy\\
	&=\int_{|y-x_0|<\eps}  \varphi^2_\eps(y)\int_{\Omega_s}\frac{ 1}{|x-y|^{N+2c}}dxdy\lesssim \eps^N\int_{|z|<1}\varphi^2(z)dz.
\end{align*}
It follows that
\[
\mathcal{E}(\varphi_\eps,\varphi_\eps) \leq \eps^{N-2r}[\varphi]^2_{r,\rr^N}+\eps^N\int_{|z|<1}\varphi^2(z)dz.
\]
Similar as in the first case we get $\alpha\leq r$.

\textbf{Step II. Behavior at infinity}
We now improve \eqref{est.1.infty.1} for large times. Observe that we deal with a   symmetric Markov semigroup. Indeed 
 the bilinear form $\mE$ associated with the operator $\ml$  is symmetric and satisfies the following: $u\in \mathcal{H}(\rr^N)$ implies $|u|\in \mathcal{H}(\rr^N)$ and
  $\mE(|u|,|u|)\leq \mE(u,u)$. This shows that the semigroup is positive \cite[Corollary 2.18, p.58]{MR2124040}. Moreover for any $u\in \mathcal{H}(\rr^N), u\geq 0$ we have $1\wedge u=\max\{1,u\}\in \mathcal{H}(\rr^N)$ and 
  $\mE(1\wedge u,1\wedge u)\leq \mE(u,u)$. In view of \cite[Corollary 2.18, p.58]{MR2124040} (se also \cite[Theorem 1.4.1, p.25]{MR2778606}), the form $\mE$ generates a symmetric Markov semigroup. 

We use inequality \eqref{ineq.energy.p}:  \[
   \|u\|^2_{L^2(\rr^N)}\leq C \big( \mE(u,u)+\|u\|_{L^1(\rr^N)}^{\frac{4r}{N+2r}} \mE(u,u) ^{\frac{N}{N+2r}}\big)\ \forall u\in \mathcal{H}(\rr^N)\cap L^1(\rr^N).
  \]
  For functions $u$ with $\|u\|_{L^1(\rr^N)}\leq 1$ we have 
  $
   \|u\|^2_{L^2(\rr^N)}\leq h(\mE(u,u))$ with 
   $ h(t)=C ( t+  t ^{\frac{N}{N+2r}} ).
  $
  This shows that 
$
  \mE(u,u)\geq h^{-1}( \|u\|^2_{L^2(\rr^N)}).
  $
  When restricting to the class of functions $u$ with their $L^2$ norm less than a constant $a$ we use the behavior of $h^{-1}(t)$ near the origin: $h^{-1}(t)\simeq t^{1+\frac{2r}N}$. Then
  \[
  \mE(u,u)\geq h^{-1}( \|u\|^2_{L^2(\rr^N)})\geq c_{r,N} ( \|u\|^2_{L^2(\rr^N)})^{\frac{N+2r}N}, \quad \forall \|u\|^2_{L^2(\rr^N)}\leq a.
  \]
  We are in the framework of \cite[Prop.~III.2, Th.~III.3]{MR1418518} with $\theta(t)=t^{1+\frac{2r}{N}}$. Indeed, from the previous $L^1-L^\infty$ property we know that there exists a positive time $t_0$ such that
  $\|S(t_0)\|_{1,2}=a$ or $\|S(t_0)\|_{1,\infty}=a^2<\infty$. It follows that we have the following decay of the semigroup:
\begin{equation}
\label{est.1.infty.2}
  \|S(t)u_0\|_{L^\infty(\rr^N)}\leq \frac{\beta}{t^{\frac N{2r}}}\|u_0\|_{L^1(\rr^N)}\quad \forall u_0\in L^1(\rr^N), \forall \ t>t_0,
\end{equation}
 where $t_0$ depends only on $a$ above.  
 Together \eqref{est.1.infty.1} and \eqref{est.1.infty.2} give the desired result.

 We now treat the optimality of the decay for large times. Let $\alpha>0$ be such the following
$$
\label{est.1.infty.0.t0}
\|S(t)u_0\|_{L^\infty(\rr^N)}\leq \frac{A}{t^{\frac N{2\alpha}}}\|u_0\|_{L^1(\rr^N)}\quad \forall u_0\in L^1(\rr^N), \forall t>t_0,
$$
holds for some positive  time
   $t_0>0$  and some constant $A>0$.
In view of \cite[Prop. III.1]{MR1418518} it follows that there exist two positive constants $C$ and $k=k(t_0)$ such that
 \[
 \|v\|_{L^2(\rr^N)}^{2+\frac{4\alpha}N}\leq C\mE(v,v), \quad \forall v\in D(\ml)\cap L^1(\rr^N), \ \|v\|_{L^1(\rr^N)}\leq 1, \ \|v\|_{L^2(\rr^N)}\leq k.
 \]
 Renormalizing it follows that we must have the following inequality
 \begin{equation}\label{normalized} 
  \|v\|_{L^2(\rr^N)}^{2+\frac{4\alpha}N}\leq C  \|v\|_{L^1(\rr^N)}^{\frac{4\alpha}N}\mE(v,v), \quad \forall v\in D(\ml)\cap L^1(\rr^N), \ \|v\|_{L^2(\rr^N)}\leq k \|v\|_{L^1(\rr^N)}.
  \end{equation}
 
  We prove that this implies $\alpha\geq r$.  Let us choose 
  $\varphi\in C_c^\infty(\rr^N)$ such that it is supported in the unit ball and  two sequences $(x_n)_{n\geq 1}\in \Omega_r$ and $(r_n)_{n\geq 1}\in \rr$ such that $r_n\rightarrow \infty$ as $n\rightarrow \infty$ and $|x_n|>>\max\{r_n, r_n^{\frac{2\alpha}{N+2r}+1}\}$. Since $\Omega_s$ is a bounded domain it implies that the balls $B_{r_n}(x_n)$ are included in $\Omega_r$. 
  We consider $\varphi_n(x)=\varphi((x-x_n)/r_n)$. It follows that $\|\varphi_n\|_{L^p(\rr^N)}\simeq r_n ^{N/p}$ for $1\leq p\leq \infty$, with  ${\rm supp}\ (\varphi_n)\subset B_{r_n}(x_n)\subset \Omega_r$ so $[\varphi_n]_{s,\Omega_s}=0$ and
  \[
  [\varphi_n]^2_{r,\Omega_r}\leq [\varphi_n]^2_{r,\rr^N}
= \int_{\rr^N}\int_{\rr^N}\frac{(\varphi((x-x_n)/r_n)-\varphi_n((y-x_n)/r_n))^2}{|x-y|^{N+2r}}dxdy\simeq
r_n^{N-2r}[\varphi]^2_{r,\rr^N}.
  \] 
  Also $ \|\varphi_n\|_{L^2(\rr^N)}\leq k \|\varphi_n\|_{L^1(\rr^N)}$ for $n$ large enough.  
  We claim that for constant $C_s$ depending on $\Omega_s$ we have
  \begin{equation}\label{ineg.termen.mixt}
 I= \int_{\Omega_{s}}\int_{\Omega_{r}}\frac{(\varphi_n(x)- \varphi_n(y))^2 }{|x-y|^{N+2c}}dydx\lesssim \frac{r_n^N}{(|x_n|-r_n-C_s)^{N+2r}}.
  \end{equation}
  Introducing function $\varphi_n$ in \eqref{normalized} we obtain for all $n$ large enough  
  \[
  (r_n^{\frac N2})^{2+\frac{4\alpha}N}\lesssim (r_n^{N})^{\frac{4\alpha}{N}} \Big(r_n^{N-2r}+\frac{r_n^N}{(|x_n|-r_n-C_s)^{N+2r}}\Big)
  \]
  and hence 
  \[
  1\lesssim r_n^{2(\alpha-r)}+\frac{r_n^{2\alpha}}{(|x_n|-r_n-C_s)^{N+2r}}.
  \]
  Since $|x_n|>>r_n^{\frac{2\alpha}{N+2r}+1}$ the last term goes to zero as $n\rightarrow \infty$. 
  This clearly implies $\alpha\geq r$.
  
  It remains to prove \eqref{ineg.termen.mixt}. Let us first observe that for $x\in \Omega_s$ and $|y-x_n|>r_n$ we have $|y-x|\geq |y|-|x|\geq |x_n|-r_n-C_s$ for some constant $C_s$ depending on $\Omega_s$. 
  Since ${\rm supp}\ (\varphi_n)\subset B_{r_n}(x_n)\subset \Omega_r$  we have
  \begin{align*}
  I=& \int_{\Omega_s}\int_{\Omega_r}\frac{  \varphi^2_n(y)dy }{|x-y|^{N+2c}}dx =\int_{\Omega_r} \varphi^2_n(y) \int_{\Omega_s}\frac{ dx }{|x-y|^{N+2c}}dy=\int_{B_{r_n}(y_n)} \varphi^2_n(y) \int_{\Omega_s}\frac{ dx }{|x-y|^{N+2c}}dy
  \\
 \leq&\int_{|y-x_n|<r_n}\varphi^2\big( \frac{y-x_n}{r_n}\big)\int_{\Omega_s} \frac{dx}{ |x_n|-r_n-C_s}dy\simeq
 \frac{r_n^N}{(|x_n|-r_n-C_s)^{N+2c}}
  \end{align*}
and the proof is finished.
  \end{proof}

\section{The asymptotic expansion of solutions}In this section we prove Theorem \ref{first.term}. The proof is quite technical and we will split it in various subsections. It is based on the method of rescaled solutions, see \cite{MR1028745,MR1977429}. To simplify the presentation  we will refer simply by $B, B^c$ to the unit ball $B_1(0)$ and its open complementary $\mathbb{R}^N\setminus \overline{B_1(0)}$, and by $B_\rho, B^c_\rho$ when the radius of the ball is $\rho$.  
For the reader's convenience  we will not make always explicit the time dependence of $u(x,t)$ unless it is necessarily.

Given a solution $u$ of problem~\eqref{eq:main} we consider
$$
u_\lambda(x,t)=\lambda^N u(\lambda x, \lambda^{2r}t)
$$
which is a solution to
\begin{equation}\label{rescaled}
\begin{cases}
\displaystyle u_{\lambda,t}(x,t)=\alpha_s\lambda^{2r-2s}\int _{B_{1/\lambda}} \frac{u_\lambda(y)-u_\lambda(x)}{|x-y|^{N+2s}}dy + \alpha_c\lambda^{2r-2c}\int _{B^c_{1/\lambda}} \frac{u_\lambda(y)-u_\lambda(x)}{|x-y|^{N+2c}}dy,\ &x\in B_{1/\lambda},\ t>0,\\

\displaystyle u_{\lambda,t}(x,t)= \alpha_r\int _{B^c_{1/\lambda}} \frac{u_\lambda(y)-u_\lambda(x)}{|x-y|^{N+2r}}dy + \alpha_c \lambda^{2r-2c}\int _{B_{1/\lambda}} \frac{u_\lambda(y)-u_\lambda(x)}{|x-y|^{N+2c}}dy,\ &x\in B^c_{1/\lambda},\ t>0,\\

u_\lambda(x,0)=\lambda^N u_0(\lambda x),
\end{cases}
\end{equation}
where for simplicity we omit the dependence on $t$.

Observe that the equation can be written as 
\[
\begin{cases}
\partial_t u_\lambda (x,t)=(\mathcal{L}_\lambda u)(x,t), & x\in \rr^N,\ t>0,\\
u_\lambda(0,x)=\lambda^N u_0(\lambda x), &x\in \rr^N,
\end{cases}
\]
where for any $\lambda>0$ we set
$$
\label{Llambda}
(\ll\varphi)(x)=
\begin{cases}
	\displaystyle \alpha_s\lambda^{2r-2s}\int_{|y|<1/\lambda}\frac{\varphi(y)-\varphi(x)}{|x-y|^{N+2s}}dy+
	\alpha_c\lambda^{2r-2c}\int_{|y|>1/\lambda}\frac{\varphi(y)-\varphi(x)}{|x-y|^{N+2c}}dy,& |x|<\frac 1\lambda,\\[20pt]
	\displaystyle	\alpha_r\int_{|y|>1/\lambda}\frac{\varphi(y)-\varphi(x)}{|x-y|^{N+2r}}dy+\alpha_c\lambda^{2r-2c}\int_{|y|<1/\lambda}
	\frac{\varphi(y)-\varphi(x)}{|x-y|^{N+2c}}dy,& |x|>\frac 1\lambda.
\end{cases}
$$

This section is divided as follows: first we obtain various estimates for operator $\ll$ and for the bilinear form associated with it $\mathcal{E}_\lambda(\varphi,\varphi)=(-\ll\varphi,\varphi)$. We use them to obtain uniform estimates for the family of rescaled solutions $(u_\lambda)_{\lambda>0}$ and use them to prove the compactness of the considered family which will converge to a profile $U_M$. We will characterize the profile $U_M$ to be the unique solution of the fractional heat equation $U_t+C_{\Omega_{r}}(-\Delta)^{r}U=0$ with initial data $M\delta_0$ taken in the sense of measures, see \cite{MR3614666}, where $C_{\Omega_{r}}$ is the constant from~\eqref{eq:constants}. Let us recall first that in \cite[Th. 9.1]{MR3614666} the authors show that for any initial data $\mu_0\in \mathcal{M}_s(\rr^N)$, the space of locally finite Radon measures 
satisfying
\[
\int_{\rr^N}(1+|x|)^{-(N+2s)}d|\mu|(x)<\infty,
\]
there exists  a unique very weak solution of the fractional heat equation in the following sense:
\begin{definition}\label{very-weak}
	We say that $u$ is a very weak solution of the equation \eqref{eq:fractional-heat} if\\
	i) $u\in L^1_{loc}(0,T:L^1(\rr^N, (1+|x|)^{-(N+2s)}))$,\\
	 ii) it satisfies the equality
	\[
	\int_0^\infty \int_{\rr^N} u(x,t)\partial_t\psi(x,t)dxdt=\int_0^\infty \int_{\rr^N} u(x,t)(-\Delta)^r\psi(x,t)dxdt
	\] 
	for all   $\psi \in C^\infty_c((0,\infty)\times \rr^N)$ and\\
	 iii) its trace is $\mu_0$ in the following sense 
	\[
	\int_{\rr^N}\psi d\mu_0 =\lim_{t\rightarrow 0^+} \int_{\rr^N} u(x,t)\psi(x)dx, \quad \text{for all}\ 
	\psi\in C_0(\rr^N).
	\]
\end{definition}
Moreover, the above very weak solutions are given by the representation formula
\[
U(x,t)=\int_{\rr^N}K^r_t(x-y) d\mu_0(y).
\]
where $K_t^r$ is a smooth function. For a collection of basic properties of the kernel $K_t^r$ we refer to \cite[Section 2.2]{MR3789847}.

\subsection{Estimates for the operator $\ll$}

We include here various lemmas used  in the proof of Theorem \ref{first.term}

\begin{lemma}\label{est.Llamba}
Let $\rho>0$ and $1/\lambda<\rho/2$. 
	For any  $\varphi\in W^{2,\infty}(\rr^d)$ such that $\varphi(x)=\varphi(0)$ in the ball $B_\rho(0)$ the following holds:
	\[
	|(\ll \varphi)(x)|\lesssim 
	\begin{cases}
		\lambda^{2r-2c} \rho^{-2c} \|\varphi\|_{L^\infty(\rr^N)}, & |x|<1/\lambda,\\[10pt]
	\rho^{2-2r} \|D^2 \varphi \|_{L^\infty(\rr^N)} +\rho^{-2r}\|\varphi\|_{L^\infty(\rr^N)},&\frac 1\lambda <|x| <\rho,\\[10pt]
	\rho^{2-2r}\|D^2 \varphi \|_{L^\infty(\rr^N)} +(\rho^{-2r}+\lambda^{2r-2c-N}\rho^{-N-2c})\|\varphi\|_{L^\infty(\rr^N)}, & |x|>\rho.
	\end{cases} 
	\]
\end{lemma}

\begin{proof}
Let us first consider the case  $|x|<1/\lambda$. Since $1/\lambda <\rho/2$ and $\varphi$ is constant in the ball $B_\rho(0)$ the first term in $(\ll \varphi)(x)$ vanishes. For the second term we use  that $\varphi(y)=\varphi(x)$ for $|y|<\delta$ and $|y-x|>\rho/2$ when $|y|>\rho$:
\begin{align*}
\int_{|y|>1/\lambda}&\frac{|\varphi(y)-\varphi(x)|}{|x-y|^{N+2c}}dy=
\int_{|y|>\rho}\frac{|\varphi(y)-\varphi(x)|}{|x-y|^{N+2c}}dy\leq 2\|\varphi\|_{L^\infty(\rr^N)}\int _{|y|>\rho}\frac 1{|y-x|^{N+2c}}dy\\
&\lesssim 
\|\varphi\|_{L^\infty(\rr^N)}\int _{|y-x|>\rho/2}\frac 1{|y-x|^{N+2c}}dy\simeq
\|\varphi\|_{L^\infty(\rr^N)} \int _{\rho/2}^{\infty} s^{-1-2a}\simeq \rho^{-2a} \|\varphi\|_{L^\infty(\rr^N)}.
\end{align*}
When $1/\rho<|x|<\rho$ and $|y|<\rho$, $\varphi(y)=\varphi(x)$ so the second term in $\ll$ vanishes. Thus
   \begin{align*}
\label{}
  (\ll \varphi)(x)&=\alpha_r\int_{|y|>1/\lambda}\frac{\varphi(y)-\varphi(x)}{|x-y|^{N+2r}}dy=
 \alpha_r \int_{|y|>\rho}\frac{\varphi(y)-\varphi(x)}{|x-y|^{N+2r}}dy\\
  &=\alpha_r \int_{\rho<|y|<2\rho}\frac{\varphi(y)-\varphi(x)}{|x-y|^{N+2r}}dy+ \alpha_r\int_{|y|>2\rho}\frac{|\varphi(y)-\varphi(x)|}{|x-y|^{N+2r}}dy.
\end{align*}
For the first term we use the fact that the gradient of $\varphi$ vanishes at the point $x$. Hence $\varphi(y)-\varphi(x)=(y-x)\nabla \varphi (x)+|y-x|^2 O(\|D^2\varphi\|_{L^\infty(\rr^N)}) =  |y-x|^2 O(\|D^2\varphi\|_{L^\infty(\rr^N)})$. We obtain that
\begin{align*}
\label{}
   |   (\ll \varphi)(x)|& \lesssim   \|D^2\varphi\|_{L^\infty(\rr^N)}\int_{\rho<|y|<2\rho}\frac{dy}{|x-y|^{N+2r-2}}+\|\varphi\|_{L^\infty(\rr^N)}\int_{|y|>2\rho}\frac{1}{|x-y|^{N+2r}}dy\\
   &\lesssim \|D^2\varphi\|_{L^\infty(\rr^N)}\int_{|z|<3\rho}\frac{dz}{|z|^{N+2r-2}}+\|\varphi\|_{L^\infty(\rr^N)}\int_{|z|<\rho}\frac{dz}{|z|^{N+2r}}\\
   &\lesssim \rho^{2-2r}\|D^2\varphi\|_{L^\infty(\rr^N)}+\rho^{-2r}\|\varphi\|_{L^\infty(\rr^N)}.
\end{align*}
Let us now consider the case $|x|>\rho$. Here we split the operator in two parts:
\begin{align*}
\label{}
  (\ll \varphi)(x)&=\alpha_r\int_{1/\lambda<|y|}\frac{\varphi(y)-\varphi(x)}{|x-y|^{N+2r}}dy+\alpha_c\lambda^{2r-2c}\int_{|y|<1/\lambda}\frac{\varphi(y)-\varphi(x)}{|x-y|^{N+2c}}\\&
  :=\alpha_r A+\alpha_c B.
\end{align*}
In the case of $A$ we use the fact $B_{\rho/2}(x)\subset \{y: |y|<1/\lambda\}$ and also that the function $(y-x)\nabla \varphi (x)$ is odd and therefore
$$
\int_{1/\lambda<|y|, |x-y|<\rho/2}\frac{(y-x)\nabla \varphi(x)}{|x-y|^{N+2r}}dy = 0.
$$
Hence
\[
A=\int_{1/\lambda<|y|, |x-y|<\rho/2}\frac{\varphi(y)-\varphi(x)-(y-x)\nabla \varphi(x)}{|x-y|^{N+2r}}dy+\int_{1/\lambda<|y|, |y-x|>\rho/2}\frac{\varphi(y)-\varphi(x)}{|x-y|^{N+2r}}dy\]
and
\begin{align*}
|A| &\lesssim  \|D^2\varphi\|_{L^\infty(\rr^N)}\int_{|y-x|<\rho/2}\frac{dy}{|x-y|^{N+2r-2}}+\|\varphi\|_{L^\infty(\rr^N)}\int_{|y-x|>\rho/2}\frac{dy}{|x-y|^{N+2r}}\\
 &\lesssim  \|D^2\varphi\|_{L^\infty(\rr^N)}\int_{|z|<\rho/2}\frac{dz}{|z|^{N+2r-2}}+\|\varphi\|_{L^\infty(\rr^N)}\int_{|z|>\rho/2}\frac{dy}{|z|^{N+2r}}\\
  &\simeq \rho^{2-2r}\|D^2\varphi\|_{L^\infty(\rr^N)}+\rho^{-2r}\|\varphi\|_{L^\infty(\rr^N)}.
\end{align*}
For the second term we use that $|x-y|>\rho-1/\lambda >\rho/2$. It satisfies
\[
  |B|\lesssim \lambda^{2r-2c} \|\varphi\|_{L^\infty(\rr^N)}\lambda ^{-N}\rho ^{-N-2c}.
\]
The proof is now finished.
\end{proof}

Let us consider a function $\psi\in C^\infty(\rr^N)$ that vanishes identically in the unit ball, $\psi(x)\equiv 1$  
for $|x|>2$ and $0\leq \psi \leq 1$. Set $\psi_R(x)=\psi(x/R)$. As a consequence of the above lemma we obtain the following result for $\psi_R$.
\begin{lemma}\label{lemma:bound-for-psi}
For any $R>2$ and $\lambda>\max\{1,1/\rho\}$ the following holds:
	\[
	|(\ll \psi_R)(x)|\lesssim 
	\begin{cases}
		\lambda^{2r-2c} R^{-2c} \|\psi\|_{L^\infty(\rr^N)}, & |x|<1/\lambda,\\[10pt]
	R^{-2r}(\|D^2 \psi \|_{L^\infty(\rr^N)} +\|\psi\|_{L^\infty(\rr^N)}),&\frac 1\lambda <|x| <\rho,\\[10pt]
	R^{-2r}\big( \|D^2 \psi \|_{L^\infty(\rr^N)} + (1+(\lambda R)^{2r-2c-N})\|\psi\|_{L^\infty(\rr^N)}\big), & |x|>\rho.
	\end{cases} 
	\]
\end{lemma}
\begin{remark}
Under the assumption that $2r-2c\leq N$ we obtain  for any $|x|>\rho$ that:
	\[
	|(\ll \psi_R)(x)|\lesssim R^{-2r}\|\psi\|_{W^{2,\infty}(\rr^N)}.
	\]
	This estimate is the same  when the fractional Laplacian of order $r$, $(-\Delta )^{r}$, acts on the rescaled function $\psi_R$ defined above.
	\end{remark}

\begin{proof}
	We apply  Lemma \ref{est.Llamba} with $\rho=R$ and $1/\lambda\leq 1<\rho/2$. The case $|x|<1/\lambda$ follows immediately. When $1/\lambda<|x|<\rho$ we use the definition of $\psi_R$  to obtain that $\|D^2 \psi_R\|_{L^\infty(\rr^N)}=R^{-2}\|D^2 \psi\|_{L^\infty(\rr^N)}$ and $|(\ll \psi_R)(x)|\lesssim R^{-2r} ( \|D^2 \psi \|_{L^\infty(\rr^N)}+ \| \psi \|_{L^\infty(\rr^N)})$. Similar arguments for $|x|>\rho$ show that
	\[
	|(\ll \psi_R)(x)|\lesssim R^{-2r} \|D^2 \psi \|_{L^\infty(\rr^N)} + R^{-2r }(1+(\lambda R)^{2r-2c-N}) \|\psi\|_{L^\infty(\rr^N)}
	\]
	and finish the proof.
\end{proof}

We now give a few estimates for the rescaled energy $\el$ associated with $u_\lambda$:
\begin{align}\notag
\label{e.lambda}
    \el(\varphi,\varphi)=&\frac{\alpha_s\lambda^{2r-2s}}4 \iint_{|x|,|y|<1/\lambda} \frac{(\varphi(y)-\varphi(x))^2}{|x-y|^{N+2s}}+\frac{\alpha_r}4
  \iint_{|x|,|y|>1/\lambda} \frac{(\varphi(y)-\varphi(x))^2}{|x-y|^{N+2r}}\\
  \nonumber&+\frac{\alpha_c}{2}
  \lambda^{2r-2c} \iint_{|x|<1/\lambda<|y|} \frac{(\varphi(y)-\varphi(x))^2}{|x-y|^{N+2c}}.
\end{align}

\begin{lemma} \label{max.est} Let us denote $m=\max\{r,s,c\}$.
	For any $\varphi\in H^m(\rr^N)$  the following holds for any $\lambda>0$:
	\begin{equation}
\label{ineg.e.1}\notag
  \el(\varphi,\varphi)\lesssim \lambda^{2r-2m}[\varphi]^2_{m,B_{2/\lambda}(0)}+[\varphi]^2_{r,B_{1/\lambda}^c}+(\lambda^{-N}+\lambda^{2r-2c-N})\int_{\rr^N}\frac{(\varphi(y)-\varphi(0))^2}{|y|^{N+2c}}dy.
\end{equation}
\end{lemma}

\begin{proof}
	For the first term we use the fact that both $x$ and $y$ are in the ball $B_{1/\lambda}(0)$:
	\begin{align*}
\label{}
 \lambda^{2r-2s}  &\iint_{|x|,|y|<1/\lambda} \frac{(\varphi(y)-\varphi(x))^2}{|x-y|^{N+2s}}dxdy = \lambda^{2r-2s} \iint_{|x|,|y|<1/\lambda} \frac{(\varphi(y)-\varphi(x))^2|x-y|^{2m-2s}}{|x-y|^{N+2m}}\\
 &\leq \lambda ^{2r-2m}\iint_{|x|,|y|<1/\lambda} \frac{(\varphi(y)-\varphi(x))^2}{|x-y|^{N+2m}}dxdy=
 \lambda ^{2r-2m}[\varphi]_{m,B_{1/\lambda}(0)}^2.
\end{align*}
We estimate  the last term. We split it in two parts: one in which $x$ and $y$ are close enough where we can use the same arguments as before and a second one where $x$ and $y$ are separated. Indeed
\begin{align*}
    \lambda^{2r-2c} &\iint_{|x|<1/\lambda<|y|} \frac{(\varphi(y)-\varphi(x))^2}{|x-y|^{N+2c}}dxdy\\
    &=  {\lambda^{2r-2c} \iint_{|x|<1/\lambda<|y|<2/\lambda} \frac{(\varphi(y)-\varphi(x))^2}{|x-y|^{N+2c}}}dxdy+  {\lambda^{2r-2c} \int_{|x|<1/\lambda} \int _{2/\lambda<|y|} \frac{(\varphi(y)-\varphi(x))^2}{|x-y|^{N+2c}}}dxdy\\
    &=A+B.
\end{align*}
In the case of $A$ the same arguments as above show that
\[
|A|\leq \lambda^{2r-2m}[\varphi]_{m,B_{2/\lambda}(0)}.
\]
In the case of $B$ we introduce the value of $\varphi$ at $x=0$ since for $\lambda$ large enough $\varphi(x)$ is close to $\varphi(0)$. We obtain 
\begin{align*}
\label{}
  |B|\leq &\lambda^{2r-2c} \int_{|x|<1/\lambda}  (\varphi(x)-\varphi(0))^2\ dx\int _{|y|>2/\lambda} \frac{dy}{|x-y|^{N+2c}}\\
  &+ \lambda^{2r-2c} 
    \int _{|y|>2/\lambda} (\varphi(y)-\varphi(0))^2\ dy \int_{|x|<1/\lambda}\frac{dx}{|x-y|^{N+2c}}\\
    &=B_1+B_2.
\end{align*}
Using that $|x|<1
/\lambda<2/\lambda<|y|$ we have $|y-x|>|y|-|x|\geq |y|/2$. Thus 
\begin{align*}
\label{}
  B_1&\lesssim \lambda^{2r-2c} \int _{|x|<1/\lambda} (\varphi(x)-\varphi(0))^2dx \int _{|y|>2/\lambda}\frac {dy}{|y|^{N+2c}}=
  \lambda^{2r} \int _{|x|<1/\lambda} (\varphi(x)-\varphi(0))^2dx\\
  &\leq  \lambda^{-N} \int _{|x|<1/\lambda} \frac{(\varphi(x)-\varphi(0))^2}{|x|^{N+2r}}dx.
\end{align*}
In the case of $B_2$ using again that $|y-x|\geq |y|/2$ we obtain 
\begin{align*}
\label{}
  B_2\lesssim \lambda^{2r-2c-N} \int _{|y|>2/\lambda}\frac{( \varphi(y)-\varphi(0))^2}{|y|^{N+2c}}dy.
\end{align*}
The proof is now finished. 
\end{proof}

\begin{lemma}\label{lemma:bound-energy}Let $r\leq c+N/2$.
 For any 
 $\rho>0$ and $\lambda>\max\{1,2/\rho\}$ the following holds:
\begin{equation}
\label{energy.est.h1}\notag
  \el(\varphi,\varphi)\leq C(\rho) 
  \|\varphi\|^2_{H ^1(|x|>\rho)}, \ \forall  \varphi \in C_c^{\infty}(|x|>\rho).
\end{equation}
\end{lemma}

\begin{proof}
We split $\el(\varphi,\varphi)$ as follows
	\begin{align*}
    \el(\varphi,\varphi)&=\frac{\alpha_s\lambda^{2r-2s}}4 \iint_{|x|,|y|<1/\lambda} \frac{(\varphi(y)-\varphi(x))^2}{|x-y|^{N+2s}}dxdy+\frac{\alpha_r}4
  \iint_{|x|,|y|>1/\lambda} \frac{(\varphi(y)-\varphi(x))^2}{|x-y|^{N+2r}}dxdy\\
  \nonumber&\quad +
 {\frac{\alpha_c  \lambda^{2r-2c}}{2}}\iint_{|x|<1/\lambda<|y|} \frac{(\varphi(y)-\varphi(x))^2}{|x-y|^{N+2c}}dxdy\\
  &=I_1+I_2+I_3.
\end{align*}
Let us consider $\varphi\in C_c^\infty(|x|>\rho)$.
 Since $\lambda>2/\rho$ the support of $\varphi$ gives us that  term $I_1$ vanishes. 
We emphasize that this term can be estimated without the support assumption. 
For second one we use the integrals in the whole space:
\[
I_2
  \leq \alpha_r\|\varphi\|_{H^1(\rr^N)}^2=\alpha_r  \|\varphi\|^2_{H ^1(|x|>\rho)}.\]
 It remains to analyze $I_3$. 
Again we use that  $\varphi(x)$ vanishes for $|x|<1/\lambda<\rho/2$, so
\begin{align*}
  I_3&\leq {\alpha_c\lambda^{2r-2c}}\iint_{|x|<1/\lambda<|y|} \frac{\varphi^2(y) }{|x-y|^{N+2c}}dxdy=\alpha_c\lambda^{2r-2c} \iint_{|x|<1/\lambda<\rho<|y|} \frac{\varphi^2(y) }{|x-y|^{N+2c}}dxdy\\
  &=\alpha_c\lambda^{2r-2c} \int _{|y|>\rho}\varphi^2(y) \int _{|x|<1/\lambda}\frac {dx}{|x-y|^{N+2c}}dy.
\end{align*}
For $\rho/2>1/\lambda$ we get $|x-y|\geq |y|-|x|>\rho-1/\lambda>\rho/2$ and 
\[
I_3\lesssim \rho^{-N-2c} \lambda ^{2r-2c-N}\int _{|y|>\rho}\varphi^2(y)dy.
\]
This finishes the proof.
\end{proof}

\subsection{Uniform estimates for the rescaled solutions} We now obtain uniform estimates for the rescaled solutions $u_\lambda=\lambda^N u(\lambda^{2r}t,\lambda x)$. We assume that $2r-2c\leq N$. Also since the decay obtained in Theorem \ref{decay} is not uniform for small and large times we need first to assume that the initial data $u_0$ belongs to the space $L^1(\rr^N)\cap L^2(\rr^N)$. This assumption will be assumed until the final step of this subsection where it will be dropped and the result of Theorem \ref{first.term} is obtained by using a density argument. 

\begin{theorem}
For any $u_0\in L^1(\rr^N)\cap L^2(\rr^N)$ the rescaled solutions $u_\lambda(x,t)$ satisfy the following uniform estimates:
\begin{equation}\label{est.p.unif}
\| u_\lambda(t)\|_{L^2(\rr^N)}\leq C(\|u_0\|_{L^1(\rr^N)}, \|u_0\|_{L^2(\rr^N)})t^{-\frac N{4r}}, \quad \forall\ t>0,
\end{equation}
\begin{equation}\label{est.e.unif}
 \mathcal{E}_\lambda(u_\lambda(t),u_\lambda(t))\leq C(\|u_0\|_{L^1(\rr^N)}, \|u_0\|_{L^2(\rr^N)}) t^{-(1+\frac N{2r})}, \quad \forall\ t>0.
\end{equation}
\end{theorem}

\begin{proof}
The first property follows from the definition of the rescaled solutions $u_\lambda$,  the decay properties in Theorem \ref{decay} and the $L^2$-contractivity of the semigroup
\begin{align*}
  \|u(t)\|_{L^2(\rr^N)}&\leq C\min\Big\{  \|u_0\|_{L^2(\rr^N)}, \|u_0\|_{L^1(\rr^N)}(t^{-\frac{N}{4r}} + t^{-\frac{N}{4\min\{r,s\}  }})\Big\}\\
  &\leq  C(\|u_0\|_{L^1(\rr^N)}, \|u_0\|_{L^2(\rr^N)})(t+1)^{-\frac{N}{4r}}.
\end{align*}
 For the second one we now use Theorem \ref{thm:existence_hilbert_CH} and the decay of the $L^2$ norm to obtain that 
\begin{align*}
\mE(u(t),u(t))\leq \frac{\|u(t/2)\|_{L^2(\rr^N)}^2}{t}\leq C(\|u_0\|_{L^1(\rr^N)}, \|u_0\|_{L^2(\rr^N)})t^{-1-\frac{N}{2r}} .
\end{align*}
The definition of $\mE$ and $\mE_\lambda$ gives us that
\[
\mathcal{E}_\lambda(u_\lambda(t),u_\lambda(t))=\lambda^{-2r-N}\mE(u(\lambda^{2r}t),u(\lambda ^{2r}t))\leq 
 C(\|u_0\|_{L^1(\rr^N)}, \|u_0\|_{L^2(\rr^N)}) t^{-(1+\frac N{2r})},
\]
which finishes the proof.
%
\end{proof}


\begin{lemma}\label{lemma_time_derivative_bound}
 For any $u_0\in L^1(\rr^N)\cap L^2(\rr^N)$ and $\rho>0$ there exists a constant $C_\rho=C(\rho,\|u_0\|_{L^1(\rr^N)},
 \|u_0\|_{L^2(\rr^N)})$ such that

$$
\|\partial_t u_\lambda\|_{L^2( (\tau,\infty) : H^{-1}(B^c_\rho)  )}\leq C_\rho\tau^{-\frac{N}{2r}}
$$
holds for all   $\tau>0$,  
uniformly in $\lambda>\max\{1,2/\rho\}$. 
\end{lemma}
\begin{proof}
Let us take $\psi\in C_c^\infty(B_\rho^c)$. We have that
$$
\|\partial_t u_\lambda\|_{H^{-1}(B^c_\rho) } = \sup <\partial_t u_\lambda, \psi>_{H^{-1}, H_0^1}\ \text{ for }\|\psi\|_{H^1_0(B_\rho^c)}  \leq 1.
$$
By Theorem~\ref{thm:existence_hilbert_CH}$\ \partial_t u_\lambda (t)\in L^2(\mathbb{R}^N)$ for any $t>0$  thus we have that this duality product can be expressed as
$$
<\partial_t u_\lambda, \psi>_{H^{-1}, H_0^1} = \int_{B_\rho^c} \partial_t u_\lambda \psi = \int_{\mathbb{R}^N} \partial_t u_\lambda \psi =\int_{\mathbb{R}^N} \mathcal{L}_\lambda u_\lambda \psi = -\el(u_\lambda,\psi).
$$
From Lemma~\ref{lemma:bound-energy} we have  $ \el(\psi,\psi)\leq C(\rho)$ for all $\lambda>\max\{1,2/\rho\}$   and this provides that
$$
\el(u_\lambda,\psi)\leq \el(u_\lambda,u_\lambda)^{1/2}  \el(\psi,\psi)^{1/2}\leq C(\rho)^{1/2}\el (u_\lambda,u_\lambda)^{1/2} 
$$
which in view of~\eqref{est.e.unif} gives that
\begin{align*}
\int_\tau^\infty  \|\partial_t u_\lambda\|^2_{H^{-1}(B^c_\rho) } &\leq C(\rho)\int_\tau^\infty \el(u_\lambda,u_\lambda)\leq C(\rho, \|u_0\|_{L^1(\rr^N)},
 \|u_0\|_{L^2(\rr^N)}) \int_\tau^\infty t^{-1-\frac{N}{2r}}dt,
\end{align*}
which finishes the proof.
\end{proof}

 \begin{lemma}\label{lemma:tail-control}Assume $r<c+N/2$. There exists $C=C(\|u_0\|_{L^1(\rr^N)}, \|u_0\|_{L^2(\rr^N)})$ and $\alpha(r,c,N)>0$ such that
  for any $R>0$
 $$
 \int_{|x|>2R}|u_\lambda(x,t)|\ dx \leq   \int_{|x|> R} u_{0}(x)\ dx +  C\left(\frac{t^{\alpha(r,c,N)}}{R^{2c}} +\frac{t}{R^{2r}}\right)
 $$
holds uniformly for all $\lambda>1$
 \end{lemma}
\begin{proof}
By the comparison principle it is sufficient to consider nonnegative solutions. Let $\psi\in C^\infty(\rr^N)$ such that $\psi(x)\equiv 0$ for $|x|<1$, $\psi(x)\equiv 1$  
for $|x|>2$ and $0\leq \psi \leq 1$. Set $\psi_R(x)=\psi(x/R)$. Multiplying the equation satisfied by $u_\lambda$ by $\psi_R$ and integrating in time we get
$$
\int_{|x|>2R} u_\lambda(x,t)dx  \leq \int_{\rr^N} u_\lambda \psi_R\ dx =\int_{\rr^N} u_{0,\lambda} \psi_R\ dx + \int_0^t \int_{\rr^N}(\mathcal{L}_\lambda u_\lambda)(x,s)\psi_R(x)ds.
$$
For  $\lambda>1$ we have
$$
\int_{\rr^N} u_{0,\lambda} \psi_R\ dx \leq \lambda^N\int_{|x|>R} u_{0}(\lambda x)\ dx = \int_{|x|>\lambda R} u_{0}(x)\ dx \leq \int_{|x|> R} u_{0}(x)\ dx
$$
so we focus now on the last integral. We split it on two parts, say
$$
 \int_{\rr^N}(\mathcal{L}_\lambda u_\lambda)(s)\psi_R dx= \int_{|x|<1/\lambda} u_\lambda (s)(\mathcal{L}_\lambda\psi_R) dx dt   + \int_{|x|>1/\lambda} u_\lambda(s) (\mathcal{L}_\lambda\psi_R) dxdt:=I(s)+II(s).
$$

Let us start with $I(s)$. Applying Lemma~\ref{lemma:bound-for-psi} for $|x|<1/\lambda$ and H\"older's inequality we get
\begin{align}\notag
|I(s)|&\leq \| u_\lambda (s)\|_{L^{1+\varepsilon}(B_{1/\lambda})}\cdot    \| \mathcal{L}_{\lambda}\psi_R\|_{L^{(1+\varepsilon)'}(B_{1/\lambda})}
\leq \| u_\lambda (s)\|_{L^{1+\varepsilon}(B_{1/\lambda})} \lambda^{2r-2c-\frac{N}{(1+\varepsilon)'}}R^{-2c} \|\psi\|_{L^\infty(\rr^N)}\\[8pt]
\notag &\leq s^{-\frac{N}{2r}\frac{\eps}{1+\eps}} \lambda^{2r-2c-\frac{N\eps}{1+\eps}}R^{-2c}
\end{align}
where $(1+\varepsilon)'= (1+\varepsilon)/\eps$. Now we can integrate in time and obtain that
$$
\int_0^t |I(s)|ds \leq  \lambda^{2r-2c-\frac{N\eps}{1+\varepsilon}}R^{-2c} t^{1-\frac{N}{2r}\frac{\varepsilon}{1+\varepsilon}}
$$
provided that $\frac{N}{2r}\frac{\varepsilon}{1+\varepsilon}<1$. Under the assumption   $2r-2c<N$ we can choose  $\varepsilon=\eps(r,c,N)$ such that
$$
\frac{2r-2c}{N} \leq \frac{\varepsilon}{1+\varepsilon} < \frac{2r}{N}
$$
and then we choose $\alpha(r,c,N)=1-\frac{N}{2r}\frac{\varepsilon}{1+\varepsilon}$.

For the second integral the mass conservation of $u_\lambda$ and the estimates for $|\mathcal{L}_\lambda \psi_R|$ from Lemma~\ref{lemma:bound-for-psi} provide
$$
\int_0^t |II(s)|ds\leq \int_0^t \int_{|x|>1/\lambda} u_\lambda |\mathcal{L}_\lambda \psi_R|dx ds \leq \frac{t\|u_0\|_{L^1(\rr^N)}} {R^{2r}} \|\psi\|_{W^{2,\infty}(\rr^N)}.
$$
This finishes the proof.
\end{proof}

\subsection{Compactness of the family $\{u_\lambda\}$}\label{sect-compactness}
 
 In this section we will show several results about the family $\{u_\lambda\}_{\lambda>0}$ that will allow us study the convergence of this family to a certain function $U$. 
 
 {\bf Step I. Compactness in $C_{loc}((0,\infty),L^1_{loc}(B_\rho^c))$.}
 Let us now choose $0<\tau<T<\infty$ and $\rho>0$. Using estimate \eqref{est.e.unif} we obtain that
 \begin{equation}\label{est.r.rho}
 \|u_\lambda\|_{L^\infty((\tau,T),H^r(B_\rho^c))}\lesssim \tau ^{-(\frac N{4r}+\frac 12)}.
 \end{equation}
 Using also that by Lemma~\ref{lemma_time_derivative_bound}  for any $\lambda>\lambda(\rho)$ we have a uniform bound for the time derivative
\[
 \|\partial_t u_\lambda\|_{L^2( (\tau,t) : H^{-1}(B^c_\rho)  )}\leq C(\rho)\tau^{-\frac{N}{2r}}
\]
 we can apply Aubin-Lions-Simon compactness argument to obtain that, up to a subsequence,
 \[
 u_\lambda\rightarrow U \ \text{in}\ C([\tau,T],L^2_{loc}(B_\rho^c)),
 \]
 and in particular we also have this convergence in $C([\tau,T],L^1_{loc}(B_\rho^c))$. By a diagonal argument we obtain the desired property.

 {\bf Step II. Compactness in $C_{loc}((0,\infty),L^1_{loc}(\rr^N))$.} Let us choose an $R>0$. We prove that $u_\lambda$ converges to $U$ in $C_{loc}((0,\infty),L^1(B_R))$. 
 
For any $\rho<R$ we have 
$$
 \int_{B_R} |u_\lambda(t) - U(t)|dx \leq \int_{B_\rho} \big(|u_\lambda(t)| + |U(t)|\big)dx +  \int_{\rho<|x|<R}  |u_\lambda(t) - U(t)|dx.
$$
For a fixed $\rho>0$ the second integral goes to 0 as $\lambda\to\infty$, so let us focus on the first one. By Hölder's inequality and estimate~\eqref{est.p.unif}
$$
 \int_{B_\rho} |u_\lambda(t)|dx\leq \| u_\lambda(t) \|^{1/2}_{L^2(\rr^N)}\cdot \rho^{N/2}\lesssim t^{-\frac{N}{8r}}\rho^{N/2}.
$$

 On the other hand, in view of estimate \eqref{est.p.unif} for each $t>0$, up to a subsequence,  $u_\lambda(t)\rightharpoonup U(t)$ in $L^2(\rr^N)$. Hence estimate \eqref{est.p.unif} transfers to $U$:
 \begin{equation}\label{eq:U_in_L2}
 \|U(t)\|_{L^2(\rr^N)}\leq C t^{-\frac{N}{4r}}.
 \end{equation}
 Hence for any $\varepsilon>0$ we can choose a small enough $\rho$ such that
$$
 \int_{B_\rho} (|u_\lambda(x,t)| + |U(x,t)|)dx  \lesssim t^{-\frac{N}{8r}}\rho^{N/2}<\frac{\varepsilon}{2}.
$$
For this $\rho$ fixed we choose $\lambda>\lambda(\rho)$
such that
$$
 \int_{\rho<|x|<R}  |u_\lambda (t)- U(t)|dx<\frac{\varepsilon}{2},
$$
proving the desired result.

{\bf Step III. Compactness in $C_{loc}((0,\infty),L^1(\rr^N))$.} The compactness in  $C((\tau,T),L^1_{loc}(\rr^N))$ follows from the previous step. The tail control proved in Lemma \ref{lemma:tail-control} shows that this convergence is also in $C((\tau,T),L^1(\rr^N))$. It means that for each $t>0$, $u_\lambda(t)\rightarrow U(t)$ in $L^1(\rr^N)$. This shows that in particular $U(t)$ conserves the mass along the time:
\begin{equation}
\label{mass.conservation.U}\notag
\int_{\rr^N}U(x,t)dx=M, \quad \forall \ t>0.
\end{equation}

 {\bf Step IV. Regularity of the profile $U$.} Estimate \eqref{est.r.rho}  gives us that $u_\lambda(t)\rightharpoonup U(t)$ in $H^r(B_\rho^c)$ and 
 \begin{equation}\label{est.r.rho.U}\notag
 \|U\|_{L^\infty((\tau,T),H^r(B_\rho^c))}\leq C(\tau).
 \end{equation}
 We will prove that in fact $U\in L^\infty( (\tau,T) : H^r(\rr^N)  )$. We already have, by~\eqref{est.e.unif}, that
$$
\int_{B_\rho^c}\int_{B_\rho^c} \frac{\big(U(x,t)-U(y,t)\big)^2}{|x-y|^{N+2r}} dxdy\leq 
C(\|u_0\|_{L^1(\rr^N)},\|u_0\|_{L^2(\rr^N)})t^{-\frac{N}{2r}-1},
$$
for any $\rho>0$. We can define now the family
$$
w_\rho(x,y):=\frac{\big(U(x,t)-U(y,t)\big)^2}{|x-y|^{N+2r}}\chi_{B_\rho^c}(x) \chi_{B_\rho^c}(y)
$$
which satisfies that $0\leq w_\rho (x,y)\leq w_{\rho'} (x,y)$ for any $\rho>\rho'$ and as $\rho \to 0$,
$$
w_\rho(x,y) \to \frac{\big(U(x,t)-U(y,t)\big)^2}{|x-y|^{N+2r}}\text{ almost everywhere in }\rr^N\times \rr^N,
$$
so  by the monotone convergence theorem and the fact, that by~\eqref{eq:U_in_L2}, $U\in L^2(\rr^N)$ we obtain that $U(t)$ belongs $H^r(\rr^N)$ and 
$$
\int_{\rr^N} \int_{\rr^N} \frac{\big(U(x,t)-U(y,t)\big)^2}{|x-y|^{N+2r}} dxdy\leq C(\|u_0\|_{L^1(\rr^N)},\|u_0\|_{L^2(\rr^N)})t^{-\frac{N}{2r}-1}.
$$
%
%

{\bf Step V. Equation satisfied by the limit $U$.}\label{sect:eq-limit}
We now show that the limit function $U$ is solution in the sense  of Definition \ref{very-weak} of the equation
\begin{equation}
	\label{eq.U}
	U_t + C_{\Omega_{r}}(-\Delta)^r U = 0 \text{ for }x\in\rr^N,\quad U(0)=\mu_0=M\delta_0
\end{equation}
where $M$ is the mass of the initial data  $u_0$. For any $M$ 
the results of \cite{MR3614666} show that there exists a unique \textit{very weak} solution $U_M$ of problem \eqref{eq.U} given by
\[
U_M(x,t)=MC_\Omega^{\frac 1{2r}} K_t^r(xC_\Omega^{\frac 1{2r}}),
\]
 where $K_t(x)$ is  the classical kernel of the fractional heat equation $U_t + (-\Delta)^r U = 0$ in $\rr^N$.

\begin{lemma}
	The limit function  $U\in C((0,\infty),L^1(\rr^N))$ satisfies 
	\begin{align*}
		\int_{\rr^N}& U(x,T)\varphi(x)dx - \int_{\rr^N} U(x,\tau)\varphi(x)dx \\
		&= \frac{\alpha_r}2\int_\tau^T \int_{\rr^N} \int_{\rr^N}\frac{U(x,t)-U(y,t)}{|x-y|^{N+2r}} \big(\varphi(x)-\varphi(y)\big)dxdydt,
	\end{align*}
	for any $\varphi\in C_c^\infty(\rr^N)$ and  $0<\tau<T<\infty.$ 
\end{lemma}

\begin{proof}
	Let us multiply  our equation by $\varphi$ to obtain, for any $0<\tau < T<\infty$,
	\begin{align*}
		\int_{\rr^N}& u_\lambda(x,T)\varphi(x)\ dx - \int_{\rr^N}  u_\lambda(x,\tau)\varphi(x)\ dx \\
		&= \int_\tau^T \int_{\rr^N} \mathcal{L}_\lambda  (u_\lambda(x,t))\varphi(x) dxdt=
		\int_\tau^T \mathcal{E}_\lambda(u_\lambda,\varphi)=I_1+I_2,
	\end{align*}
	where
	\begin{align*}
		I_1=&\frac{\alpha_s}2 \lambda^{2r-2s}\int_\tau^T \int_{B_{1/\lambda}} \int_{B_{1/\lambda}}\frac{u_\lambda(x,t)-u_\lambda(y,t)}{|x-y|^{N+2s}} \big(\varphi(x)-\varphi(y)\big)dxdydt\\
		&+\alpha_c
		\lambda^{2r-2c}\int_\tau^T \int_{B_{1/\lambda}} \int_{B^c_{1/\lambda}}\frac{u_\lambda(x,t)-u_\lambda(y,t)}{|x-y|^{N+2c}} \big(\varphi(x)-\varphi(y)\big)dxdydt
	\end{align*}
and
\[
I_2=\frac{\alpha_r}2 \int_\tau^T \int_{B^c_{1/\lambda}} \int_{B^c_{1/\lambda}}\frac{u_\lambda(x,t)-u_\lambda(y,t)}{|x-y|^{N+2r}} \big(\varphi(x)-\varphi(y)\big)dxdydt.
\]
Under the assumption $2r-2c<N$ estimate  \eqref{est.e.unif} and similar estimates as in the proof of  Lemma \ref{max.est},  show  that $I_1\rightarrow 0$ as $\lambda\rightarrow \infty$. We prove now that 
\begin{equation}\label{term.I_2}
I_2\rightarrow \frac{\alpha_r}2\int_\tau^T \int_{\rr^N}\int_{\rr^N}\frac{(U(x,t)-U(y,t))(\varphi(x)-\varphi(y))}{|x-y|^{N+2r}}dxdydt.
\end{equation}

We set 
$
\psi\in L^2(\rr^N \times \rr^N)$ by 
\[
\psi(x,y):=\frac{\varphi(x)-\varphi(y)}{|x-y|^{(N+2r)/2}},\]
 and 
  \[
  f_\lambda (x,y,t):= \frac{(u_\lambda(x,t)-u_\lambda(y,t))}{
  |x-y|^{(N+2r)/2}}\chi_{B_{1/\lambda}^c}(x)\chi_{B_{1/\lambda}^c}(y).\]
   Clearly, by~\eqref{est.r.rho} we have that the family $f_\lambda$ is uniformly bounded in $ L^2((\tau,T)\times\rr^N \times \rr^N)$ and so there exists a $f\in  L^2((\tau,T)\times\rr^N \times \rr^N)$ such that $f_\lambda\rightharpoonup f$.
We will prove that
 \[
f(x,y, t)=\frac{(U(x,t)-U(y,t)\big)}{|x-y|^{(N+2r)/2}},\ \text{a.e. in}\ (\tau,T)\times \rr^N \times \rr^N
\]
so \eqref{term.I_2} holds.

For a bounded domain $\Omega$ the strong convergence obtained in the previous steps  show that, up to a subsequence,  $u_\lambda(x,t)\to U(x,t)$ almost everywhere in $(\tau,T)\times\Omega$, which means that $f_\lambda(x,y,t) \to \big(U(x,t)-U(y,t)\big)|x-y|^{-(N+2r)/2}$ a.e. in $(\tau,T)\times\big(\Omega\times\Omega\big)$. 
Classical results on weak convergence 
(see for example \cite[Prop. 5.4.6, p.~106]{MR3112778}) show  that $f_\lambda \rightharpoonup \big(U(x,t)-U(y,t)\big)|x-y|^{-(N+2r)/2}$ in $L^2(\Omega\times\Omega)$, which means precisely that $f(x,y,t)=\big(U(x,t)-U(y,t)\big)|x-y|^{-(N+2r)/2}$ in $(\tau, T)\times \big(\Omega\times\Omega\big)$. Since $\Omega$ was arbitrary, we are done.
\end{proof}
 
 Next we focus on the initial data.  
 \begin{lemma}
For any $\varphi\in BC(\rr^N)$ 
\begin{equation}
\label{limit.initial.data}
  \lim_{t\rightarrow 0} \int_{\rr^N } U(x,t)\varphi(x)dx=M\varphi(0).
\end{equation}
\end{lemma}

\begin{remark}
	In fact in order to use the uniqueness results of \cite{MR3614666} it is sufficient to consider the above limit only for functions   $\varphi\in C_0(\rr^N)$, i.e. functions that vanish at infinity. Nonetheless we present here a stronger result.
\end{remark}

\begin{proof}
 Let us first consider a smooth function $\varphi$ that it is constant in a ball centered at the origin $B_\rho(0)$. Multiplying equation \eqref{rescaled} by $\varphi$ and integrating in the space variable gives us
 \begin{align*}
 \int_{\rr^N} u_\lambda(x,t)\varphi(x)dx&-\int_{\rr^d}
 u_{0,\lambda}(x)\varphi(x)dx=\int_0^t \int_{\rr^N}u_\lambda(s,x)(\mathcal{L}_\lambda \varphi)(x)dxds\\
 &=\int_0^t \int_{|x|<1/\lambda} u_\lambda(s,x)(\mathcal{L}_\lambda \varphi)(x)dxds+ \int_0^t \int_{|x|>1/\lambda}u_\lambda(s,x)(\mathcal{L}_\lambda \varphi)(x)dxds\\
 &:=I_1+I_2.
 \end{align*}
For $\lambda\geq \max\{1,2/\rho\} $  we can apply Lemma \ref{est.Llamba} and the fact that $r\leq c+\frac N2$ to estimate the second term
\[
|I_2|\lesssim Mt \Big(\rho^{2-2r}\|D^2\varphi\|_{L^\infty(\rr^d)} + (\rho^{-2r} +\rho^{-N-2c})\|\varphi\|_{L^\infty(\rr^N)} \Big).
\]
In the case of $I_1$ we proceed as for the tail control. Let $\varepsilon>0$ be a small parameter that will be chosen latter. We use the first estimates in Lemma \ref{est.Llamba}:
\begin{align*}
|I_1| &\leq \|(\mathcal{L}_\lambda \varphi)\|_{L^{(1+\eps)'}(|x|<1/\lambda)}\int_0^t \| u_\lambda(s)\|_{L^{1+\eps}(\rr^N)}   ds\\
&\lesssim 
\lambda^{2r-2c}\rho^{-2c}\|\varphi\|_{L^\infty(\rr^N)}\lambda^{-N(1+\eps)'}\int_0^t
s^{-\frac{N\eps}{2r(1+\eps)}}ds\\
&\simeq\rho^{-2c}\|\varphi\|_{L^\infty(\rr^N)}\lambda^{2r-2c-\frac{N\eps}{1+\eps}}t^{1-\frac{N\eps}{2r(1+\eps)}},
\end{align*}
provided that $\frac{N\eps}{2r(1+\eps)}<1$. Under the assumption $2r-2c<N$ we can choose $\eps=\eps(r,c,N)>0$ such that $2r-2c\leq \frac{N\eps}{1+\eps}$ and $\frac{N\eps}{2r(1+\eps)}<1$. Thus there exists a positive constant $\alpha(r,c,N)$ such that 
\[
|I_1|\leq \rho^{-2c}\|\varphi\|_{L^\infty(\rr^N)} t^{\alpha(r,c,N)}.
\]
It follows that for large enough $\lambda$ we have
\[
\Big|\int_{\rr^N} u_\lambda(x,t)\varphi(x)dx-\int_{\rr^d}
 u_{0,\lambda}(x)\varphi(x)dx\Big|\leq
 C(\varphi)(t+t^{\alpha(r,c,N)}).
\]
Letting $\lambda\rightarrow\infty$ we obtain the same property for the limit point $U$:
\begin{equation}
\label{est.U.2}
  \Big|\int_{\rr^N} U(x,t)\varphi(x)dx-M\varphi(0)\Big|\leq
 C(\varphi)(t+t^{\alpha(r,c,N)}).
\end{equation}

This shows that \eqref{limit.initial.data} holds for functions $\varphi\in W^{2,\infty}(\rr^d)$ which are locally constants near the origin. 

Let us now show that property \eqref{limit.initial.data} holds for all $\varphi\in BC(\rr^N)$. Indeed by Lemma \ref{w2app} for any $\varphi\in BC(\rr^N)$ there exists a sequence of functions $\varphi_n\in W^{2,\infty}(\rr^N)$ which are locally constant in a neighborhood of the origin such that $\varphi_n\rightarrow \varphi$ uniformly on compact sets and $\|\varphi_n\|_{L^\infty(\rr^N)}\leq \|\varphi\|_{L^\infty(\rr^N)}$. 

For completeness we prefer to write the full argument here even though it is a standard procedure. Let us choose $R$ large enough that will be fixed latter. We write
\begin{align*}
\int_{\rr^N} U(x,t)\varphi(x)dx&-M\varphi(0)=\int_{|x|>2R} U(x,t)(\varphi(x)-\varphi_n(x))dx\\
&-M(\varphi(0)-\varphi_n(0))+\int_{|x|<2R} U(x,t)(\varphi(x)-\varphi_n(x))dx\\
&+ \int_{\rr^N} U(x,t)\varphi_n(x)dx-M\varphi_n(0)\\
&=I+II+III.
\end{align*}  
Lemma \ref{lemma:tail-control} and the strong convergence  in $L^1(\rr^N)$ of $u_\lambda(t)$ towards $U(t)$ show that limit point $U$ satisfies the same tail control  as in Lemma \ref{lemma:tail-control}:
\[ \int_{|x|>2R}|U(x,t)|\ dx \leq   \int_{|x|> R} |u_{0}(x)| dx +  C\left(\frac{t^{\alpha(r,c,N)}}{R^{2c}} +\frac{t}{R^{2r}}\right). \]
It implies that given any $\eps>0$ the first term satisfies
\begin{align*}
	|I|&\leq 2\|\varphi\|_{L^\infty(\rr^N)}\int _{|x|>2R}|U(x,t)|dx\\
	&\leq 2\|\varphi\|_{L^\infty(\rr^N)}\Big( \int_{|x|> R} |u_{0}(x)| dx +  C\big(\frac{t^{\alpha(r,c,N)}}{R^{2c}} +\frac{t}{R^{2r}}\big) \Big)<\eps,
\end{align*}
for 
$t<1$ and $R>R(\eps, u_0)$. For this large $R$ we can choose an $n$ large enough such that $II$ is small. Indeed using the uniform convergence of $\varphi_n$ towards $\varphi$ in the ball of radius $2R$ we get for large $n$ that
\[
| II |\leq M\|\varphi_n-\varphi\|_{L^\infty(|x|<2R)}<\eps.
\]
For this $n$ large enough we apply estimate \eqref{est.U.2} so we can choose a small $t$ such that $|III|<\eps$. It follows that $|I+II+III|<3\eps$ for small enough $t$ which proves the desired estimate \eqref{limit.initial.data} for $U$.
 \end{proof}

\textbf{Final Step and proof of Theorem \ref{first.term}}. For simplicity let us introduce the rescaled kernel
\[\tilde K_t^r(x)=
C_\Omega^{\frac 1{2r}} K_t^r(C_\Omega^{\frac 1{2r}}x).
\]
We first observe that it is sufficient to prove the result for initial data in $L^1(\rr^N)\cap L^2(\rr^N)$. 
Assume that the result is true for initial data in $L^1(\rr^N)\cap L^2(\rr^N)$. For given $u_0\in L^1(\rr^N)$ we choose a sequence $u_{0n}\in L^1(\rr^N)\cap L^2(\rr^N)$ such that $u_{0n}\rightarrow u_0$ in $L^1(\rr^N)$. For $t>1$ using the decay in Theorem \ref{decay} to $u-u_n$
we get
\begin{align*}
t^{\frac N{2r}(1-\frac 1p)}&\|u(t)-U_M(t)\|_{L^p(\rr^N)}\\
&\leq  t^{\frac N{2r}(1-\frac 1p)}\|u(t)-u_n(t)\|_{L^p(\rr^N)}+t^{\frac N{2r}(1-\frac 1p)}\Big\|u_n(t)-\int_{\rr^N}u_{0n}\tilde K_t^r\Big\|_{L^p(\rr^N)}\\
&\quad \quad +\Big|\int _{\rr^N}u_0-\int_{\rr^N}u_{0n}\Big|t^{\frac N{2r}(1-\frac 1p)} \|\tilde K_t^r\|_{L^p(\rr^N)}\\
&\leq 2\|u_{0}-u_{0n}\|_{L^1(\rr^N)} +t^{\frac N{2r}(1-\frac 1p)}\Big\|u_n(t)-\int_{\rr^N}u_{0n}\tilde K_t^r\Big\|_{L^p(\rr^N)}.
\end{align*}
Given an $\varepsilon>0$ we first choose an large $n$ such that the first term in the right hand side is less than $\varepsilon$ and then for any large time $t$ we obtain the desired result. 

In the case of initial data in $L^1(\rr^N)\cap L^2(\rr^N)$ all the estimates in the previous steps holds and then we obtain that up to a subsequence $u_\lambda\rightarrow U_M(t) $ in $C((0,\infty),L^1(\rr^N))$. Since in Step V we uniquely  identified the profile  $U_M$ it means that the convergence holds for the whole sequence, not only for  a subsequence.


 The $L^1$ convergence of $u_\lambda(1)$ towards $U_M(1)$ shows the desired property in $L^1$:
 \[
 \|u(t)-U_M(t)\|_{L^1(\rr^N)}\rightarrow 0.
 \]
 For $1<p<\infty$ we use the decay in $L^{2p}(\rr^N)$ norm of the solution and of the function $U_M$ (\cite[Lemma 2.2]{MR3789847}) and H\"older interpolation with exponent $\alpha=2(p-1)/(2p-1)$:
 \[
  \|u(t)-U_M(t)\|_{L^p(\rr^N)}\leq  \|u(t)-U_M(t)\|_{L^1(\rr^N)}^{ {1-\alpha} } 
  \|u(t)-U_M(t)\|_{L^{2p}(\rr^N)}^{\alpha} \leq o(1) t^{-\frac N{2r}(1-\frac 1p)}.
 \]
\section{Comments on possible extensions}
The authors would like to address here on some possible extensions of the results presented here.
\begin{itemize}
	\item[-] Theorems~\ref{decay} and~\ref{first.term} concern the case where $\Omega_s$ is a bounded Lipschitz domain. Nonetheless we believe that these results can be extended with some effort to the case where $\Omega_s$ is a finite union of sets of this kind or even a Lipschitz domain with finite measure. On the other hand the Lipschitz and finite measure hypothesis seem much harder to generalize.
	
	\item[-] Speaking of different domains, one can study Theorems~\ref{decay} and~\ref{first.term} in the case where $\Omega=\Omega_{r}\cup \Omega_{s}$ is a cilinder. This is a problem of high interest in the literature which combines in some sense the Cauchy and Dirichlet problems. Again if $\Omega_s$ is a bounded Lipschitz domain similar results are expected along the "long" variables, but  since some of our arguments imply taking ball of bigger and bigger radii inside $\Omega_{s}$ we expect some new conditions on the exponents to appear. Also, the presence of a boundary where $u=0$ may influence the long-time behaviour heavily, so this case seems challenging.
	
	\item[-] As shown in~\cite{MR4114263} it is reasonable to consider other combinations of diffusion like, for example, the case where the diffusion in the bounded domain is given by a convolution with an $L^1$ probability kernel. As long as the gradient flow structure is preserved and the different scaling properties of the operators are considered with care the authors believe that many of the results here can be extended to these cases.
	
	\item[-] Finally, it is also very interesting to study if one can recover solutions of the classical heat equation by taking the limit of the solution $u$ as the fractionary exponents $r,s,c$ go to 1. Again this question was addressed in~\cite{MR4114263} and the authors expect a positive answer to this question. 
\end{itemize}

\section{Appendix}\label{sect-appendix}
\setcounter{lemma}{0}
\setcounter{theorem}{0}
\setcounter{equation}{0}
\setcounter{subsection}{0}
\renewcommand{\theequation}{A.\arabic{equation}}
\renewcommand{\thesection}{A}
\begin{lemma}
\label{w2app}
For any function $\varphi\in BC(\rr^N)$ there exists an approximation sequence $\{\varphi_n\}_{n\geq 1}\in W^{2,\infty}(\rr^N)$ such that all $\varphi_n$ are constant in a neighborhood of the origin and satisfy:\\
i) $\|\varphi_n\|_{L^\infty(\rr^N}\leq \|\varphi\|_{L^\infty(\rr^N})$,\\
ii) $\varphi_n$ converges to $\varphi$ unifomly on compact sets. 
\end{lemma}

\begin{proof}
Let us choose a sequence of mollifiers $(\rho_n)_{n\geq 1}$ as in \cite[Ch.~4.4, p.~108]{MR2759829} and $\psi_n=\rho_n\ast \varphi$. It follows that $\psi_n$ are smooth, $\|\psi_n\|_{L^\infty(\rr^N)}\leq \|\varphi\|_{L^\infty(\rr^N)}$ and $\psi_n\rightarrow\varphi$ uniformly on compact sets \cite[Prop.~4.2.1, Ch.~4, p.~108]{MR2759829}. It remains to make the approximation to be locally constant near origin. To do that we choose a function $\theta\in C^\infty_c(\rr^N)$, $0\leq \theta\leq 1$, such that $\theta(x)\equiv 1$ in $|x|<1$ and $\theta(x)\equiv 0$ in $|x|>2$ and set $\theta_\rho(x)=\theta(x/\rho)$. We consider 
\[\varphi_{n,\rho}(x)=
\psi_n(x)+\theta_\rho(x)(\psi_n(0)-\psi_n(x)).
\]
 It follows that $\varphi_{n,\rho}$ is smooth and   $\varphi_{n,\rho}(x)=\psi_n(0)$ for $|x|<\rho$. Since $0\leq \theta_\rho\leq 1$ it follows that the first property is satisfied. For the second property let us observe that the difference between $\varphi_{n,\rho}$ and $\varphi$ satisfies
 \[
 |\varphi_{n,\rho}(x)-\varphi(x)|\leq |\psi_n(x)-\varphi(x)|+|\theta_\rho(x)||\psi_n(x)-\psi_n(0)|.
 \]
 Let us choose a compact set $K$. For any $\eps>0$ we choose an $N_\eps$ such that for any $n\geq N_\eps$,  $|\psi_n(x)-\varphi(x)|\leq \eps $ for any $x\in K$. Since for each $n\geq n_\eps$ the function $\psi_n$ is continuous we can choose $\rho_n$ small enough such that $|\psi_n(x)-\psi_n(0)|\leq \eps$ for all $|x|<2\rho_n$.  Then $ |\varphi_{n,\rho}(x)-\varphi(x)|\leq 2\eps$ and setting $\varphi_n=\varphi_{n,\rho_n}$ we finish the proof.
\end{proof}

\subsection{ Sobolev's  inequality for the fractional Laplacian  in exterior domains.
\\}\label{sec:GNS}

In this section we will consider a   domain $\Omega$ 
which satisfies the following measure density condition \eqref{measure.condition}: there exists a positive constant $C_\Omega$ such that 
\begin{equation}\
\label{measure.condition.appendix}
  |\Omega\cap B_\rho(x)|\geq C_\Omega \rho^N, \quad\forall x\in \Omega,  \forall \rho>0.
\end{equation}
In \cite{MR3280034} it has been proved  that the above  condition satisfied for all $\rho \in (0,1)$ is   equivalent to the fact that $\Omega$ is a $W^{s,p}(\Omega)$-extension domain. In particular using \cite[Th. 2.4]{MR2944369} we have that $C_c^\infty(\rr^N)$ is dense in $H^s(\Omega)$.

\begin{lemma}
	Let $\Omega$ be a domain in $\rr^N$ that  satisfies \eqref{measure.condition.appendix}. For any $1<p<\infty$ there exists a positive constant $C(N,p,s,\Omega) $ such that
	\begin{equation}
\label{ineg.E}\notag
  \int_{\Omega\setminus E}\frac{dy}{|x-y|^{N+ps}}\geq C(N,p,s,\Omega) |E|^{-ps/N}
\end{equation}
holds for all $x\in \Omega$ and for all   nonempty measurable sets $E\subset \Omega$  with finite measure.
\end{lemma}

\begin{remark}When $\Omega$ is the whole space $\rr^N$ this has been proved in \cite[Lemma A.1]{MR3133422}.
	When $x\in {\rm Int}(\Omega\setminus E)$ the above integral is infinite. The estimate is then useful when $x\in E$. 
\end{remark}


\begin{proof}
	Let us fix $x\in \Omega$ and consider $\rho$ which we will fix later.  
	We have
		\begin{align*}
\label{}
  \int_{\Omega\setminus E}\frac{dy}{|x-y|^{N+ps}}&=  \int_{(\Omega\setminus E)\cap B_\rho(x)}\frac{dy}{|x-y|^{N+ps}}+  \int_{(\Omega\setminus E)\cap (\Omega\setminus B_\rho(x))}\frac{dy}{|x-y|^{N+ps}}\\
  &\geq  \int_{(\Omega\setminus E)\cap B_\rho(x)}\frac{dy}{\rho^{N+ps}}+ \int_{(\Omega\setminus E)\cap (\Omega\setminus B_\rho(x))}\frac{dy}{|x-y|^{N+ps}}\\
  &=\frac{1}{\rho^{N+ps}}|(\Omega\setminus E)\cap B_\rho(x)|+ \int_{(\Omega\setminus E)\cap (\Omega\setminus B_\rho(x))}\frac{dy}{|x-y|^{N+ps}}.
\end{align*}
Let $\rho=\rho(x,E, \Omega)$ be such that $|\Omega\cap B_\rho(x)|=|E|$. This is possible since $x\in \Omega$, $\Omega$ is an open unbounded set so the map $\rho\in (0,\infty)\mapsto |\Omega\cap B_\rho(x)|=\int_\Omega 1_{B_\rho(x)}(y)dy\in (0,\infty)$ is continuous by dominated convergence theorem  and satisfies $\lim _{\rho\downarrow 0}|\Omega\cap B_\rho(x)|=0$ and $\lim_{\rho\rightarrow\infty} |\Omega\cap B_\rho(x)|=|\Omega|=\infty$. Then
\begin{align*}
  |  (\Omega\setminus E)\cap B_\rho(x)|&=|\Omega \cap B_\rho(x)|-|E\cap B_\rho (x)|\\
  &=|E|-|E\cap B_\rho (x)|=|E\cap (\Omega\setminus B_\rho(x))\|.
\end{align*}
Hence
\begin{align*}
  \int_{\Omega\setminus E}\frac{dy}{|x-y|^{N+ps}}&
  \geq \int_{E\cap (\Omega\setminus B_\rho(x))} \frac{dy}{\rho^{N+ps}}+ \int_{(\Omega\setminus E)\cap (\Omega\setminus B_\rho(x))}\frac{dy}{|x-y|^{N+ps}}\\
  &\geq \int_{E\cap (\Omega\setminus B_\rho(x))} \frac{dy}{|x-y|^{N+ps}}+ \int_{(\Omega\setminus E)\cap (\Omega\setminus B_\rho(x))}\frac{dy}{|x-y|^{N+ps}}\\
&= \int_{ \Omega\setminus B_\rho(x)} \frac{dy}{|x-y|^{N+ps}}.
\end{align*}
We claim now that under the assumption  \eqref{measure.condition} there exists an $\alpha=\alpha(\Omega,N)>0$ such that 
\begin{equation}
\label{cond.alpha}
  |\Omega \cap \{y:r<|y-x|<\alpha r \}|\geq \frac{r^N}2,\quad \forall r>0.
\end{equation}
Indeed 
\begin{align*}
\label{}
  |\Omega \cap \{y:r<|y-x|<\alpha r \}|&=|\Omega\cap  B_{\alpha r}(x) |-|\Omega\cap  B_{ r}(x)|\geq C_\Omega (\alpha r)^N -|B_{ r}(x)|\\
  &=C_\Omega(\alpha r)^N-\omega_Nr^N=r^N(C_\Omega\alpha^N-\omega_N)\geq \frac{r^N}2
\end{align*}
where we choose  $\alpha>1$  large enough.

 Using \eqref{cond.alpha} we obtain 
 \begin{align*}
  \int_{ \Omega\setminus B_\rho(x)}\frac{dy}{|x-y|^{N+ps}}&\geq \sum_{k\geq 0}\int _{\Omega\cap \{ \rho \alpha^k\leq |x-y|\leq \rho\alpha^{k+1}  \} }\frac{dy}{|x-y|^{N+ps}}\\
  &\geq \sum_{k\geq 0} |\Omega\cap \{ \rho \alpha^k\leq |x-y|\leq \rho\alpha^{k+1}  \}| (\rho \alpha^{k+1})^{-(N+ps)}\\
  &\geq \frac12\sum_{k\geq 0} (\rho\alpha^k)^N(\rho \alpha^{k+1})^{-(N+ps)}\\
  &=\frac 12\rho^{-sp}\alpha^{-(N+ps)}\sum_{k\geq 0}\alpha^{-psk}=C(N,p,s,\Omega)|E|^{-sp/N},
\end{align*}
%
which finishes the proof.
\end{proof}

\begin{lemma}
	\label{GNS.ext.balls} Let $\Omega$ be a domain satisfying \eqref{measure.condition.appendix}, $s\in (0,1)$ and $2s<N$. Then for any $f\in H^s(\Omega)$ we have
	$$
		\label{gns}
	  \|f\|_{L^{2^*}(\Omega)}\leq C(\Omega,N,s) [f]_{s,\Omega}.
	$$
\end{lemma}

\begin{remark}
	When $s=1$   Sobolev's inequality in  locally Lipschitz exterior domains have been considered in \cite[Lemma~3.1]{MR2109950} and \cite[Th.~II.6.1, p.~88]{MR2808162}. 
	A proof in the whole space $\rr^N$ without using Fourier analysis tools has been done in \cite[Theorem~7.1]{MR2873236}.
\end{remark}

\begin{proof}Let us denote $A_k=\{x\in \Omega: |f|>2^k\}$ for any integer $k$. 
Repeating the  arguments in the Appendix of \cite{MR2873236} we prove the estimate for functions $f:\rr^N\rightarrow\rr$ bounded measurable and compactly supported in $\rr^N$. 
Since $\Omega$ satisfies hypothesis \eqref{measure.condition}  
 we have that $C_c^\infty(\rr^N)$ is dense in  $H^s(\Omega)$ \cite{MR2944369} and by density 
we obtain the desired inequality.
\end{proof}

%
%

\subsection{Exterior $L^p$ Nash inequality for the fractional Laplacian.
\\}
\label{sec:Nash}
%
%
%
%
%
%
%
%

In view of possible applications of our results to nonlinear problems we give here a $L^p$ version of the Nash inequality proved in Lemma \ref{lemma.nash.2}. This will be useful in the context of nonlinear problems.

\begin{theorem} Let  $N\geq 1$,  $\Omega\in \rr^N$ satisfying hypothesis \eqref{measure.condition.appendix}, $r\in (0,1)$ and  $p\in (1,\infty)$. For any $f\in L^1(\Omega)$ such that $|f|^{p/2}\in H^{r}(\Omega)$ the following holds
\begin{equation}
\label{nash}\notag
  \|f\|_{L^p(\Omega)}^{\frac{p(N(p-1)+2r)}{N(p-1)}}\leq C(p,r,N) \|f\|_{L^1(\Omega)}^{\frac{2rp}{N(p-1)}}[|f|^{p/2}]_{r,\Omega}^{2}.
\end{equation}
\end{theorem}
\begin{remark}
The above condition $|f|^{p/2}\in H^{r}(\Omega)$ holds for example if $f\in L^p(\rr^N)$ with $\ml f\in L^p(\rr^N)$.  Indeed
\begin{align*}
  \int_{\Omega}\int_{\Omega}	&	\frac{||f|^{p/2}(y)-|f|^{p/2}(x)|^2}{|x-y|^{N+2r}}dxdy\lesssim \int_{\Omega}	\int_{\Omega}	
	\frac{(f(y)-f(x))(|f|^{p-2}f(y)-|f|^{p-2}f(x))}{|x-y|^{N+2r}}dxdy\\
	&\leq \mathcal{E}(f,|f|^{p-2}f)=(\mathcal L f, |f|^{p-2}f).
\end{align*}
Since $f\in L^p(\rr^N)$ it follows that  $|f|^{p-2}f$ to $L^{p'}(\Omega)$. It follows that the term $(\mathcal L f, |f|^{p-2}f)$ is finite so  $[|f|^{p/2}]_{r,\Omega}$ is well defined. 

\end{remark}

\begin{proof}
We distinguish two cases. The first one concerns the case when $2r<N$ and the inequality of Lemma~\ref{GNS.ext.balls} holds. The second one treats the remaining case $N=1$. 

Under the assumption $2r<N$ we use the inequality of Lemma~\ref{GNS.ext.balls} to obtain 
\[
[|f|^{p/2}]_{r,\Omega}\geq \| |f|^{p/2} \|_{2^*,\Omega}= \| f\|_{\frac{pN}{N-2r}}^{p/2}.
\]
Using interpolation we have
\[
\|f\|_{L^p(\Omega)}\leq \|f\|_{L^{\frac{pN}{N-2r}}(\Omega)}^\theta \|f\|_{L^1(\Omega)}^{1-\theta},
\]
where
\[ \frac 1p=\frac{\theta}{{\frac{pN}{N-2r}}}+\frac {1-\theta}1, \ i.e. \ \theta=\frac{N(p-1)}{N(p-1)+2r}.
\]
Putting together the last two inequalities we get 
\[
[|f|^{p/2}]_{r,B^c}^2\geq \Big( \frac{\|f\|_p}{\|f\|_1^{1-\theta}} \Big)^{\frac p\theta},
\]
and replacing the value of $\theta$ we obtain the desired inequality. 

It remains to consider the case when $N=1$. We use a trick that has been used previously in \cite{MR3809118}, namely we use the fact that $r/2<N=1$ and apply the previous inequality with $r/2$ instead of $r$:
  \[\|f\|_{L^p(\Omega)}^{\frac{p(N(p-1)+r)}{N(p-1)}}\leq C(p,r,N) \|f\|_{L^1(\Omega)}^{\frac{rp}{N(p-1)}}[|f|^{p/2}]_{r/2,\Omega}^{2}.
\]
We claim that for any open set $\Omega\subset \rr^N$
\begin{equation}
\label{claim.1}
  \|v\|_{r/2,\Omega}\leq  C(r) \|v\|_{L^2(\Omega)}^{1/2}[v]^{1/2}_{r,\Omega}.
\end{equation}
If we use this with $v=|f|^{p/2}$ we get 
\begin{align*}
\label{}
  \|f\|_{L^p(\Omega)}^{\frac{p(N(p-1)+r)}{N(p-1)}}\lesssim \|f\|_{L^1(\Omega)}^{\frac{rp}{N(p-1)}}[|f|^{p/2}]_{r/2,\Omega}^{2}\lesssim  \|f\|_{L^1(\Omega)}^{\frac{rp}{N(p-1)}} \|f\|_{L^p(\Omega)}^{p/2}
  [|f|^{p/2}]_{r,\Omega}
\end{align*}
which after simplifying the $L^p$-norm from the right hand side is exactly our desired inequality. 

It remains to prove claim \eqref{claim.1}. When $\Omega=\rr^N$ it is easily obtained using the Fourier transform. When $\Omega$ is an arbitrary open set of $\rr^N$ we proceed as follow:
\begin{align*}
\label{}
  [v]_{r/2,\Omega}^{2}&=\int_\Omega\int_{\Omega}
  \frac{|v(x)-v(y)|^2}{|x-y|^{N+r}}dxdy\\
  &=\iint_{x,y\in \Omega, |x-y|<\delta}
  \frac{|v(x)-v(y)|^2}{|x-y|^{N+r}}dxdy+
  \iint_{x,y\in \Omega, |x-y|>\delta}
  \frac{|v(x)-v(y)|^2}{|x-y|^{N+r}}dxdy\\
  &\leq \delta^s\iint_{x,y\in \Omega, |x-y|<\delta}
  \frac{|v(x)-v(y)|^2}{|x-y|^{N+2r}}dxdy
  + \iint_{x,y\in \Omega, |x-y|>\delta}
  \frac{|v(x)|^2+|v(y)|^2}{|x-y|^{N+r}}dxdy\\
  &\leq \delta^r[v]_{r,\Omega}^2 +2\int_\Omega v^2(x) \int_{|x-y|>\delta} \frac {dy}{|x-y|^{N+r}}dx\\
  &\leq \delta^r[v]_{r,\Omega}^2 + C(r)\delta ^{-r}\|v\|_{L^2(\Omega)}^2.
\end{align*}
Choosing $\delta^r=\|v\|_{L^2(\Omega)}/[v]_{r,\Omega}$   the claim is proved.  
\end{proof}

\bibliographystyle{plain}
\bibliography{biblio}

\begin{thebibliography}{10}

\bibitem{MR2299447}
Alexander Bendikov and Patrick Maheux.
\newblock Nash type inequalities for fractional powers of non-negative
  self-adjoint operators.
\newblock {\em Trans. Amer. Math. Soc.}, 359(7):3085--3097, 2007.

\bibitem{MR3614666}
Matteo Bonforte, Yannick Sire, and Juan~Luis V\'{a}zquez.
\newblock Optimal existence and uniqueness theory for the fractional heat
  equation.
\newblock {\em Nonlinear Anal.}, 153:142--168, 2017.

\bibitem{MR3809118}
Cristina Br\"{a}ndle and Arturo de~Pablo.
\newblock Nonlocal heat equations: regularizing effect, decay estimates and
  {N}ash inequalities.
\newblock {\em Commun. Pure Appl. Anal.}, 17(3):1161--1178, 2018.

\bibitem{MR2759829}
Haim Brezis.
\newblock {\em Functional analysis, {S}obolev spaces and partial differential
  equations}.
\newblock Universitext. Springer, New York, 2011.

\bibitem{MR3990737}
Ha\"{\i}m Brezis and Petru Mironescu.
\newblock Where {S}obolev interacts with {G}agliardo-{N}irenberg.
\newblock {\em J. Funct. Anal.}, 277(8):2839--2864, 2019.

\bibitem{MR898496}
E.~A. Carlen, S.~Kusuoka, and D.~W. Stroock.
\newblock Upper bounds for symmetric {M}arkov transition functions.
\newblock {\em Ann. Inst. H. Poincar\'{e} Probab. Statist.}, 23(2,
  suppl.):245--287, 1987.

\bibitem{MR1691574}
Thierry Cazenave and Alain Haraux.
\newblock {\em An introduction to semilinear evolution equations}, volume~13 of
  {\em Oxford Lecture Series in Mathematics and its Applications}.
\newblock The Clarendon Press, Oxford University Press, New York, 1998.
\newblock Translated from the 1990 French original by Yvan Martel and revised
  by the authors.

\bibitem{MR1418518}
Thierry Coulhon.
\newblock Ultracontractivity and {N}ash type inequalities.
\newblock {\em J. Funct. Anal.}, 141(2):510--539, 1996.

\bibitem{MR2109950}
Francesca Crispo and Paolo Maremonti.
\newblock An interpolation inequality in exterior domains.
\newblock {\em Rend. Sem. Mat. Univ. Padova}, 112:11--39, 2004.

\bibitem{MR2944369}
Eleonora Di~Nezza, Giampiero Palatucci, and Enrico Valdinoci.
\newblock Hitchhiker's guide to the fractional {S}obolev spaces.
\newblock {\em Bull. Sci. Math.}, 136(5):521--573, 2012.

\bibitem{MR2597943}
Lawrence~C. Evans.
\newblock {\em Partial differential equations}, volume~19 of {\em Graduate
  Studies in Mathematics}.
\newblock American Mathematical Society, Providence, RI, second edition, 2010.

\bibitem{MR2778606}
Masatoshi Fukushima, Yoichi Oshima, and Masayoshi Takeda.
\newblock {\em Dirichlet forms and symmetric {M}arkov processes}, volume~19 of
  {\em De Gruyter Studies in Mathematics}.
\newblock Walter de Gruyter \& Co., Berlin, extended edition, 2011.

\bibitem{MR3642095}
Ciprian~G. Gal and Mahamadi Warma.
\newblock Nonlocal transmission problems with fractional diffusion and boundary
  conditions on non-smooth interfaces.
\newblock {\em Comm. Partial Differential Equations}, 42(4):579--625, 2017.

\bibitem{MR2808162}
G.~P. Galdi.
\newblock {\em An introduction to the mathematical theory of the
  {N}avier-{S}tokes equations}.
\newblock Springer Monographs in Mathematics. Springer, New York, second
  edition, 2011.
\newblock Steady-state problems.

\bibitem{MR4114263}
A.~Gárriz, F.~Quirós, and J.~D. Rossi.
\newblock Coupling local and nonlocal evolution equations.
\newblock {\em Calc. Var. Partial Differential Equations}, 59(4):24pp, 2020.

\bibitem{MR3789847}
Liviu~I. Ignat and Diana Stan.
\newblock Asymptotic behavior of solutions to fractional diffusion-convection
  equations.
\newblock {\em J. Lond. Math. Soc. (2)}, 97(2):258--281, 2018.

\bibitem{MR1028745}
S.~Kamin and J.~L. V\'{a}zquez.
\newblock Fundamental solutions and asymptotic behaviour for the
  {$p$}-{L}aplacian equation.
\newblock {\em Rev. Mat. Iberoamericana}, 4(2):339--354, 1988.

\bibitem{Kri}
Denis Kriventsov.
\newblock Regularity for a local-nonlocal transmission problem.
\newblock {\em Arch. Ration. Mech. Anal.}, 217(3):1103--1195, 2015.

\bibitem{MR2124040}
El~Maati Ouhabaz.
\newblock {\em Analysis of heat equations on domains}, volume~31 of {\em London
  Mathematical Society Monographs Series}.
\newblock Princeton University Press, Princeton, NJ, 2005.

\bibitem{MR2873236}
Ovidiu Savin and Enrico Valdinoci.
\newblock Density estimates for a nonlocal variational model via the {S}obolev
  inequality.
\newblock {\em SIAM J. Math. Anal.}, 43(6):2675--2687, 2011.

\bibitem{MR3133422}
Ovidiu Savin and Enrico Valdinoci.
\newblock Density estimates for a variational model driven by the {G}agliardo
  norm.
\newblock {\em J. Math. Pures Appl. (9)}, 101(1):1--26, 2014.

\bibitem{MR1977429}
Juan~Luis V\'{a}zquez.
\newblock Asymptotic beahviour for the porous medium equation posed in the
  whole space.
\newblock {\em J. Evol. Equ.}, 3(1):67--118, 2003.
\newblock Dedicated to Philippe B\'{e}nilan.

\bibitem{MR3112778}
Michel Willem.
\newblock {\em Functional analysis}.
\newblock Cornerstones. Birkh\"{a}user/Springer, New York, 2013.
\newblock Fundamentals and applications.

\bibitem{MR3280034}
Yuan Zhou.
\newblock Fractional {S}obolev extension and imbedding.
\newblock {\em Trans. Amer. Math. Soc.}, 367(2):959--979, 2015.

\end{thebibliography}
\end{document}